\documentclass[reqno,12pt]{amsart}
\usepackage{hyperref}
\usepackage[latin1]{inputenc}
\usepackage{amsmath,amsthm}
\usepackage{amssymb}
\usepackage{mathabx}
\newtheorem{Theorem}{Theorem } [section]
\newtheorem{lemma}[Theorem]{Lemma}
\newtheorem{corollary}[Theorem]{Corollary}

\numberwithin{equation}{section}

\newcommand{\n}{|\hspace{-.5mm}\|}

\font\ff=cmsy10
\def\tiF{\text{\ff F\kern 0pt}{\;}^{ -1}}
\def\tF{\text{\ff F\kern 0pt}}

\begin{document}
\title[]{The Zakharov-Kuznetsov equation in weighted sobolev spaces}
\author{Eddye Bustamante, Jos\'e Jim\'enez and Jorge Mej\'{\i}a}

\address{Eddye Bustamante M., Jos\'e Jim\'enez U., Jorge Mej\'{\i}a L. \newline
Departamento de Matem\'aticas\\Universidad Nacional de Colombia\newline
A. A. 3840 Medell\'{\i}n, Colombia}
\email{eabusta0@unal.edu.co, jmjimene@unal.edu.co, jemejia@unal.edu.co}

\begin{abstract}
In this work we consider the initial value problem (IVP) associated to the two dimensional Zakharov-Kuznetsov equation $$\left. \begin{array}{rl} u_t+\partial_x^3 u+\partial_x \partial_y^2 u +u \partial_x u &\hspace{-2mm}=0,\qquad\qquad (x,y)\in\mathbb R^2,\; t\in\mathbb R,\\ u(x,y,0)&\hspace{-2mm}=u_0(x,y). \end{array} \right\}$$ We study the well-posedness of the IVP in the weighted Sobolev spaces $$H^s(\mathbb R^2) \cap L^2((1+x^2+y^2)^{r} dx dy),$$ with $s,r\in\mathbb R$.
\end{abstract}

\maketitle

\section{Introduction}

In this article we consider the initial value problem (IVP) associated to the two dimensional Zakharov-Kuznetsov (ZK) equation,
\begin{align}
\left. \begin{array}{rl}
u_t+\partial_x^3 u+\partial_x \partial_y^2 u +u \partial_x u &\hspace{-2mm}=0,\qquad\qquad (x,y)\in\mathbb R^2,\; t\in\mathbb R,\\
u(x,y,0)&\hspace{-2mm}=u_0(x,y).
\end{array} \right\}\label{ZK}
\end{align}
This equation is a bidimensional generalization of the Korteweg-de Vries (KdV) equation and in three spatial dimensions was derived by Zakharov and Kuznetsov in \cite{ZK1974} to describe unidirectional wave propagation in a magnetized plasma. A rigorous justificacion of the ZK equation from the Euler-Poisson system for uniformly magnetized plasma was done by Lannes, Linares and Saut in the chapter 10 of \cite{LLS2013}.

Lately, different aspects of the ZK equation and its generalizations have been extensively studied.

With respect to the local and global well posedness (LWP and GWP) of the IVP \eqref{ZK} in the context of classical Sobolev spaces, Faminskii in \cite{F1995}, established GWP in  $H^s(\mathbb R^2)$, for $s\geq 1$, integer. For that, Faminskii followed the arguments developed by Kenig, Ponce and Vega for the Korteweg-de Vries equation in \cite{KPV1993}, which use the local smoothing effect, a maximal function estimate and a Strichartz type inequality, for the group associated to the linear part of the equation, to obtain LWP by the contraction mapping principle. Then the global result is a consequence of the conservation of energy. In \cite{LP2009}, Linares and Pastor refined Faminskii's method and obtained LWP for initial data in Sobolev spaces $H^s(\mathbb R^2)$, for $s>3/4$. Recently, symmetrizing the ZK equation and using the Fourier restriction norm method (Bourgain's spaces, see \cite{B1993}), Grünrock and Herr in \cite{GH2014} improved the previous results, establishing LWP of the IVP \eqref{ZK} in $H^s(\mathbb R^2)$ for $s>1/2$. The same result of Grünrock and Herr was obtained, independently, by Pilod and Molinet in \cite{MP2014}.

LWP and GWP of the IVP \eqref{ZK} for the ZK equation and its generalizations also have been considered in the articles \cite{BL2003}, \cite{F2008}, \cite{LP2011}, \cite{LPS2010}, \cite{RV2012}, \cite{ST2010} and references therein.

In \cite{K1983}, Kato studied the IVP for the generalized KdV equation in several spaces, besides the classical Sobolev spaces. Among them, Kato considered weighted Sobolev spaces.

In this work we will be concerned with the well-posedness of the IVP \eqref{ZK} in weighted Sobolev spaces. This type of spaces arises in a natural manner when we are interested in determining if the Schwartz space is preserved by the flow of the evolution equation in \eqref{ZK}.

Some relevant nonlinear evolution equations as the KdV equation, the non-linear Schrödinger equation and the Benjamin-Ono equation, have also been studied in the context of weighted Sobolev spaces (see \cite{FLP2012}, \cite{FLP2013}, \cite{I1986}, \cite{I2003}, \cite{J2013}, \cite{N2012}, \cite{NP2009} and \cite{NP2011} and references therein).

We will study real valued solutions of the IVP \eqref{ZK} in the weighted Sobolev spaces
$$Z_{s,r}:=H^s(\mathbb R^2) \cap L^2((1+x^2+y^2)^{r} dx dy),$$
with $s,r\in\mathbb R$. 

The relation between the indices $s$ and $r$ for the solutions of the IVP \eqref{ZK} can be found, after the following considerations, contained in the work of Kato: suppose we have a solution $u\in C([0,\infty); H^s(\mathbb R^2))$ to the IVP \eqref{ZK} for some $s> 1$. We want to estimate $(pu,u)$, where $p(x,y):=(1+x^2+y^2)^r$ and $(\cdot,\cdot)$ is the inner product in $L^2(\mathbb R^2)$. Proceeding formally we multiply the ZK equation by $up$, integrate over $(x,y)\in \mathbb R^2$ and apply integration by parts to obtain:
$$\dfrac d{dt} (pu,u)=-3(p_x \partial_x u, \partial_x u)-(p_x \partial_y u, \partial_y u)-2(p_y \partial_yu,\partial_xu)+((p_{xxx}+p_{xyy})u,u)+\dfrac 23 (p_x u^3,1).$$

To see that $(pu,u)$ is finite and bounded in $t$, we must bound the right hand side in the last equation in terms of $(pu,u)$ and $\| u\|_{H^s}^2$. The most significant terms to control in the right hand side in the equation are the three first ones. They may be controlled in the same way. Let us indicate how to bound the first term. Using the Interpolation Lemma \ref{interpol} (see section 2), for $\theta\in [0,1]$ and $u\in Z_{s,r}$ we have
$$\| (1+x^2+y^2)^{(1-\theta)r/2} u \|_{H^{\theta s}}\leq C \|  (1+x^2+y^2)^{r/2} u\|_{L^2}^{(1-\theta)}\| u\|^\theta_{H^s}.$$
The term $3(p_x\partial_x u,\partial _x u)$ can be controlled when $\theta s=1$ if
\begin{align}
|p_x| \leq (1+x^2+y^2)^{(1-\theta)r}.\label{intro1}
\end{align}

Taking into account that $|p_x|\leq (1+x^2+y^2)^{r-1/2} $, in order to have \eqref{intro1} it is enough to require that $r-1/2=(1-\theta)r$. This condition, together with $\theta s=1$, leads to $r=s/2$.

In this way the natural weighted Sobolev space to study the IVP \eqref{ZK} is $Z_{s,s/2}$.

Our aim in this article is to prove that the IVP \eqref{ZK} is LWP in $Z_{s,s/2}$ for $s>3/4$, s real. In order to do that we consider two cases: (i) $3/4<s\leq 1$ and (ii) $s> 1$.

\begin{enumerate}
\item[(i)] In the first case $(3/4<s\leq 1)$ we symmetrize the equation as it was done by Grünrock and Herr in \cite{GH2014}. In this manner we can establish the estimates for the group associated to the linear part of the symmetrization of the ZK equation, using directly the correspondent estimates for the group associated to the linear KdV equation. In particular, the method used by Faminskii in \cite{F1995}, in order to obtain an estimate for the maximal function associated to the group of the linear ZK equation, is simpler in the case of the linear symmetrized ZK equation. In fact, Faminskii's method, in this case, combines in a transparent way the decay in $t$ of the fundamental solution of the linear KdV equation with the procedure followed by Kenig, Ponce and Vega in \cite{KPV1991}, to obtain the maximal type estimate for the KdV equation.\\
On the other hand, we need a tool to treat fractional powers of $(|x|+|y|)$. A key ingredient in this direction is a characterization of the generalized Sobolev space
\begin{align}
L^p_b(\mathbb R^n):=(1-\Delta)^{-b/2}L^p(\mathbb R^n),\label{intro2}
\end{align}
due to Stein (see \cite{S1961} and \cite{S1970}) (when $p=2$, $L^2_b(\mathbb R^n)=H^b(\mathbb R^n)$). This characterization is as follows.\\
\textbf{Theorem A}. \textit{Let $b\in(0,1)$ and $2n/(n+2b)\leq p<\infty$. Then $f\in L^p_b(\mathbb R^n)$ if and only if}
\begin{enumerate}
\item[(a)] $f\in L^p(\mathbb R^n)$, \textit{and}
\item[(b)] $\mathcal D^b f(x):=\left( \displaystyle\int_{\mathbb R^n} \dfrac{|f(x)-f(y)|^2}{|x-y|^{n+2b}} dy \right)^{1/2}\in L^p(\mathbb R^n)$,
\end{enumerate}
\textit{with}
\begin{align}
\| f\|_{L^p_b}:=\|(1-\Delta)^{b/2}f \|_{L^p}\simeq \| f\|_{L^p}+\| D^b f\|_{L^p}\simeq \| f\|_{L^p}+\| \mathcal D^b f\|_{L^p}, \label{intro3}
\end{align}
\textit{where $D^s f$ is the homogeneous fractional derivative of order $b$ of $f$, defined through the Fourier transform by}
\begin{align}
(D^s f)^\wedge(\xi) =|\xi|^b \hat f(\xi), \label{intro4}
\end{align}
\textit{($\xi\in\mathbb R^n$ is the dual Fourier variable of $x\in\mathbb R^n$).}\\
From now on we will refer to $\mathcal D^b f$ as the Stein derivative of $f$.\\
As a consequence of Theorem A, Nahas and Ponce proved (see Proposition 1 in \cite{NP2009}) that  for measurable functions $f,g:\mathbb R^n\to \mathbb C$:
\begin{align}
\mathcal D^b(fg) (x)\leq & \| f\|_{L^\infty} (\mathcal D^b g)(x)+|g(x)| \mathcal D^b f(x),\; a.e. \,x\in\mathbb R^n, \, \text{and} \label{intro5}\\
\| \mathcal D^b(fg)\|_{L^2} \leq &\| f\mathcal D^b g \|_{L^2}+ \|g\mathcal D^b f \|_{L^2}.\label{intro6}
\end{align} 
It is unknown whether or not \eqref{intro6} still holds with $D^b$ instead of $\mathcal D^b$.\\
Following a similar procedure to that done by Nahas and Ponce in \cite{NP2009}, in order to obtain a pointwise estimate for $\mathcal D^b(e^{it|x|^2})(x)$ (see Proposition 2 in \cite{NP2009}), we get to bound appropriately $\mathcal D^b(e^{itx_1^3})(x_1,x_2)$ for $b\in (0,1/2]$ (see Lemma \ref{lemastein1} in section 2).\\
Using \eqref{intro3} (for $p=2$), \eqref{intro5}, \eqref{intro6} and Lemma \ref{lemastein1} we deduce an estimate for the weighted $L^2$-norm of the group associated to the linear part of the symmetrization of the ZK equation, $\| (|x|+|y|)^b V(t)f\|_{L^2}$, in terms of $t$, $\|(|x|+|y|)^b f\|_{L^2}$ and $\| f\|_{H^{2b}}$ (see Corollary 2.7 in section 2).\\
This estimate is similar to that, obtained by Fonseca, Linares and Ponce in \cite{FLP2014} (see formulas 1.8 and 1.9 in Theorem 1) for the KdV equation.\\
The linear estimates for the group of the linear part of the symmetrization of the ZK equation, together with the estimate for the weighted $L^2$-norm of the group, allow us to obtain LWP of the IVP \eqref{ZK} in a certain subspace of $Z_{s,s/2}$ by the contraction mapping principle.

\item[(ii)] In the second case ($s>1$) we use the LWP of the IVP \eqref{ZK} in $H^s(\mathbb R^2)$, obtained by Linares and Pastor in \cite{LP2009}. Then we perform a priori estimates on the ZK equation in order to prove that if the initial data belongs to $Z_{s,s/2}$ then necessarily $u\in L^\infty([0,T];L^2((1+x^2+y^2)^{s/2} dx dy))$. In this step of the proof we apply the interpolation inequality (Lemma \ref{interpol} in section 2), mentioned before, which was proved in \cite{FP2011}. Finally, we conclude the proof of the LWP in $Z_{s,s/2}$ in a similar manner as it was done in \cite{BJM2013} for a fifth order KdV equation.
\end{enumerate}

Now we formulate in a precise manner the main result of this article.

\begin{Theorem} \label{maint} Let $s>3/4$ and $u_0\in Z_{s,s/2}$ a real valued function. Then there exist $T>0$ and a unique $u$, in a certain subspace $Y_T$ of $C([0,T];Z_{s,s/2})$, solution of the IVP \eqref{ZK}.
(The definition of the subspace $Y_T$ will be clear in the proof of the theorem).

Moreover, for any $T'\in(0,T)$ there exists a neighborhood $V$ of $u_0$ in $Z_{s,s/2}$ such that the data-solution map $\tilde u_0 \mapsto \tilde u$ from $V$ into $Y_{T'}$ is Lipschitz.

When $3/4<s\leq 1$, the size of $T$ depends on $\| u_0\|_{Z_{s,s/2}}$, and when $s>1$ the size of $T$ depends only on $\|u_0 \|_{H^s}$.
\end{Theorem}

This article is organized as follows: in section 2 we establish some linear estimates for the group associated to the linear part of the symmetrization of the ZK equation (subsection 2.1), we recall the Leibniz rule for fractional derivatives, deduced by Kenig, Ponce and Vega in \cite{KPV1993} and an interpolation lemma proved in \cite{FP2011} and \cite{NP2009} (subsection 2.2), and we find (subsection 2.3) an appropriate estimate for the Stein derivative of order $b$ in $\mathbb R^2$ of the symbol $e^{itx_1^3}$ (Lemma \ref{lemastein1}), which has an important consequence (Corollary \ref{c2.7}) that affirms that the weighted Sobolev space $Z_{s,s/2}$ remains invariant by the group. In section 3, we use the results, obtained in section 2, in order to prove Theorem \ref{maint}.

Throughout the paper the letter $C$ will denote diverse constants, which may change from line to line, and whose dependence on certain parameters is clearly established in all cases.

Finally, let us explain the notation for mixed space-time norms. For $f:\mathbb R^2\times [0,T]\to\mathbb R$ (or $\mathbb C$) we have
$$\|f \|_{L^p_x L^q_{Ty}}:=\left( \int_{\mathbb R} \left( \int_\mathbb R  \int_0^T |f(x,t)|^q dt dy \right)^{p/q} dx \right )^{1/p}.$$
When $p=\infty$ or $q=\infty$ we must do the obvious changes with \textit{essup}. Besides, when in the space-time norm appears $t$ instead of $T$, the time interval is $[0,+\infty)$.

\section{Preliminary Results}

\subsection{Linear Estimates} In this section we consider the linear IVP
\begin{align}
\left. \begin{array}{rl}
v_t+\partial_x^3 v+\partial_y^3 v&\hspace{-2mm}=0,\qquad\qquad (x,y)\in\mathbb R^2,\; t\in\mathbb R\\
v(x,y,0)&\hspace{-2mm}=v_0(x,y).
\end{array}\right\}\label{pr2.1}
\end{align}

The solution of \eqref{pr2.1} is given by
\begin{align}
v(x,y,t)=[V(t)v_0](x,y),\quad (x,y)\in\mathbb R^2,\quad t\in\mathbb R,\label{pr2.2}
\end{align}

where $\{V(t)\}_{t\in\mathbb R}$ is the unitary group, defined by
\begin{align}
[V(t)v_0](x,y)=\dfrac 1{2\pi}\int_{\mathbb R^2} e^{i[t(\xi^3+\eta^3)+x\xi+y\eta]} \widehat v_0(\xi,\eta) d\xi d\eta.\label{group}
\end{align}

For $0\leq \varepsilon\leq 1/2$, let us consider the oscillatory integrals
\begin{align}
I_t(x,y):=&\int_{\mathbb R^2} |\xi|^\varepsilon e^{i[t(\xi^3+\eta^3)+x\xi+y\eta]} d\xi d\eta,\quad \text{and}\label{pr2.3}\\
J_t(x,y):=&\int_{\mathbb R^2} |\eta|^\varepsilon e^{i[t(\xi^3+\eta^3)+x\xi+y\eta]} d\xi d\eta.\label{pr2.4}
\end{align}

From lemma 2.2 in \cite{KPV1991} it follows that
\begin{align}
|I_t(x,y)|=\left| \int_\mathbb R |\xi|^\varepsilon e^{i(t\xi^3+x\xi)} d\xi\right| \left|\int_\mathbb R e^{i(t\eta^3+y\eta)} d\eta \right|\leq \dfrac C{|t|^{(\varepsilon+1)/3}} \cdot \dfrac C{|t|^{1/3}}=\dfrac C{|t|^{(2+\varepsilon)/3}}.\label{pr2.5}
\end{align}

In a similar manner, we have
\begin{align}
|J_t(x,y)|\leq\dfrac C{|t|^{(2+\varepsilon)/3}}.\label{pr2.6}
\end{align}

Proceeding as in \cite{LP2009}, from the estimates \eqref{pr2.5} and \eqref{pr2.6} we can obtain the following Strichartz-type estimates for the group.
\begin{lemma} (Strichartz type estimates). For $\varepsilon\in(0,1/2]$,
\begin{align}
\|V(\cdot_t)f \|_{L^2_T L^\infty_{xy}}&\leq C T^\gamma \|D_x^{-\varepsilon/2}f \|_{L^2_{xy}},\label{pr2.7}\\
\|V(\cdot_t)f \|_{L^2_T L^\infty_{xy}}&\leq CT^\gamma \|D_y^{-\varepsilon/2}f \|_{L^2_{xy}},\label{pr2.8}
\end{align}
where $\gamma=\dfrac{(1-\varepsilon)}6$. (Let us recall that if $(\xi,\eta)$ is the dual Fourier variable of $(x,y)\in\mathbb R^2$, $(D_x^{-\varepsilon/2} f)^\wedge (\xi,\eta):=|\xi|^{-\varepsilon/2} \widehat f (\xi,\eta)$ and $(D_y^{-\varepsilon/2} f)^\wedge (\xi,\eta):=|\eta|^{-\varepsilon/2} \widehat f (\xi,\eta)$).
\end{lemma}

In the next two lemmas we establish estimates of local type and maximal type.

\begin{lemma} (Local type estimates). There exists a constant $C$ such that
\begin{align}
\| \partial_x V(\cdot_t)v_0\|_{L_x^\infty L^2_{ty}}\leq C \|v_0 \|_{L^2_{xy}},\label{pr2.9}
\end{align}
and,
\begin{align}
\| \partial_y V(\cdot_t)v_0\|_{L_y^\infty L^2_{tx}}\leq C \|v_0 \|_{L^2_{xy}}.\label{pr2.10}
\end{align}
\end{lemma}

\begin{proof}
We only prove \eqref{pr2.9}, being the proof of \eqref{pr2.10} similar. Recall that
$$[V(t)v_0](x,y)=\dfrac 1{2\pi}\int_{\mathbb R^2} e^{i(\xi x+\eta y)} e^{it(\xi^3+\eta^3)} \widehat v_0 (\xi,\eta) d\xi d\eta.$$

Performing in the former integral the change of variables
$$\xi':=\xi^3+\eta^3,\quad \eta'=\eta,$$
we obtain
$$[V(t)v_0](x,y)=C\int_{\mathbb R^2} e^{it\xi'} e^{iy\eta'} e^{i\xi (\xi',\eta')x}\widehat v_0(\xi(\xi',\eta'),\eta')\dfrac 1{3[\xi(\xi',\eta')]^2} d\xi' d\eta'.$$

Applying Plancherel's theorem with respect to the variables $t$ and $y$, it follows that for all $x$,
\begin{align*}
\|[V(\cdot_t)v_0](x,\cdot_y) \|_{L^2_{ty}}&=\|e^{i\xi(\cdot_{\xi'},\cdot_{\eta'})x} \dfrac{\widehat v_0(\xi(\cdot_{\xi'},\cdot_{\eta'}),\cdot_{\eta'})}{3(\xi(\cdot_{\xi'},\cdot_{\eta'}))^2} \|_{L^2_{\xi'\eta'}}\\
&=\left(\int_{\mathbb R^2}\dfrac {|\widehat v_0 (\xi(\xi',\eta'),\eta')|^2}{9(\xi(\xi',\eta'))^4}  d\xi' d\eta'\right)^{1/2}.
\end{align*}

Now, we perform in the last integral the change of variables
$$\xi=\xi(\xi',\eta'),\quad \eta=\eta',$$
to obtain
\begin{align*}
\| [V(\cdot_t)v_0](x,\cdot_y)\|_{L^2_{ty}}&=\left(\int_{\mathbb R^2} \dfrac{|\widehat v_0 (\xi,\eta)|^2}{9\xi^4} 3\xi^2 d\xi d\eta \right)^{1/2}=\left(\int_{\mathbb R^2} \dfrac{|\widehat v_0 (\xi,\eta)|^2}{3\xi^2} d\xi d\eta \right)^{1/2}.
\end{align*}

Using this equality we can conclude that
\begin{align}
\| \partial_x V(t)v_0 \|_{L_x^\infty L_{ty}^2}&=\|V(t)\partial_x v_0 \|_{L_x^\infty L_{ty}^2}=\left( \int_{\mathbb R^2} \dfrac{|(\partial_x v_0)^\wedge (\xi,\eta)|^2}{3\xi^2} d\xi d\eta \right)^{1/2}\\
&=\left( \int_{\mathbb R^2} \dfrac{\xi^2|\widehat v_0 (\xi,\eta)|^2}{3\xi^2} d\xi d\eta \right)^{1/2}\leq C \| v_0 \|_{L^2_{xy}}.
\end{align}
\end{proof}

\begin{lemma} (Maximal type estimates). Let $v_0\in H^s(\mathbb R^2)$, for some $s>3/4$. Then for all $T>0$
\begin{align}
\| V(\cdot_t)v_0 \|_{L_x^2 L_{yT}^\infty}\leq C_s (1+T)^{1/2} \|D^s v_0 \|_{L_{xy}^2} \label{pr2.11}
\end{align}
and,
\begin{align}
\| V(\cdot_t)v_0 \|_{L_y^2 L_{xT}^\infty}\leq C_s (1+T)^{1/2} \|D^s v_0 \|_{L_{xy}^2}. \label{pr2.12}
\end{align}
\end{lemma}

\begin{proof}
By the symmetry of the equation $\partial_t v+\partial_x^3 v+\partial_y^3 v=0$ in $x$ and $y$, it is enough to establish estimate \eqref{pr2.11}.

Following Faminskii in \cite{F1995}, let $\mu\in C^\infty(\mathbb R)$ a nondecreasing function such that $\mu(\xi)=0$ for $\xi\leq 0$, $\mu(\xi)=1$ for $\xi\geq 1$, and $\mu(\theta)+\mu(1-\theta)=1$ for $\theta\in[0,1]$, and let us consider the sequence of functions $\{\psi_k \}_{k\in \mathbb N \cup \{0\}}$ in $C^\infty(\mathbb R^2)$, defined by
$$\psi_0(\xi,\eta):=\mu(2-|\xi|) \mu(2-|\eta|),$$
and for $k\geq 1$,
$$\psi_k(\xi,\eta):=\mu(2^{k+1}-|\xi|)\mu(2^{k+1}-|\eta|)\mu(|\eta|-2^k+1)+\mu(2^{k+1}-|\xi|)\mu(|\xi|-2^k+1)\mu(2^k-|\eta|).$$
It can be seen that for all $(\eta,\xi)\in\mathbb R^2$
\begin{align}
\sum_{k=0}^\infty \psi_k(\xi,\eta)=1.\label{mte3}
\end{align}
For $k=0,1,2,\dots$ let us define
$$I_k(x,y,t):=\int_{\mathbb R^2} e^{i[t(\xi^3+\eta^3)+x\xi+y\eta]} \psi_k(\xi,\eta)d\xi d\eta.$$
Let us estimate the oscillatory integrals $I_k(x,y,t)$. For that we procceed as Faminskii in \cite{F1995} (Lemma 2.2) and Kenig, Ponce and Vega in \cite{KPV1991} (Proposition 2.6).

\textit{Estimation of $I_0(x,y,t)$}.
\begin{align*}
|I_0(x,y,t)|&=\left| \int_{\mathbb R} e^{i(t\eta^3+y\eta)}\mu(2-|\eta|) d\eta \right| \left| \int_{\mathbb R} e^{i(t\xi^3+x\xi)} \mu(2-|\xi|) d\xi \right|\leq 4 \left| \int_{\mathbb R} e^{i(t\xi^3+x\xi)} \mu(2-|\xi|) d\xi \right|
\end{align*}

Let us define the phase function $\varphi$ by $\varphi(\xi):=t\xi^3+x\xi$. For $t\in[0,T]$, $|\xi|\leq 2$ and $|x|>48T$,
$$|\varphi'(\xi)|=|3t\xi^2+x|\geq \dfrac{|x|}2.$$

Integrating by parts, it can be shown that for $t\in [0,T]$ and $|x|>48T$,
$$\left|\int_\mathbb R e^{i\varphi(\xi)} \mu(2-|\xi|) d\xi \right|=\left|\int_\mathbb R e^{i\varphi(\xi)} \left[\dfrac 1{\varphi'(\xi)} \left(\dfrac{\mu(2-|\xi|)}{\varphi'(\xi)} \right)' \right]' d\xi \right|\leq \dfrac C{x^2}.$$

For $0<T\leq 1$, if we define
$$H_0(x):=\left\{ \begin{array}{ccc} 16 & \text{if} & |x|\leq 48,\\
\dfrac{4C}{x^2} & \text{if}& |x|>48,  \end{array}\right.$$
then for all $(x,y)\in\mathbb R^2$ and $t\in[0,T]$
$$|I_0(x,y,t)|\leq H_0(x),\quad \text{and}\quad \|H_0 \|_{L^1_x}\leq C.$$

For $T>1$, let us define
$$H_0(x):=\left\{ \begin{array}{ccc} 16 & \text{if} & |x|\leq 48T,\\
\dfrac{4C}{x^2} & \text{if}& |x|>48T.  \end{array}\right.$$

Then, for all $(x,y)\in\mathbb R^2$ and $t\in[0,T]$,
\begin{align}
|I_0(x,y,t)|\leq H_0(x),\quad \text{and}\quad \| H_0\|_{L^1_x} \leq C(1+T).\label{mte4}
\end{align}

In this manner we can conclude that, for $T>0$, there exists $H_0\in L^1(\mathbb R)$ such that for all $(x,y)\in\mathbb R^2$ and $t\in[0,T]$ the assertion \eqref{mte4} holds.

\textit{Estimation of $I_k(x,y,t)$, $k\geq 1$}.

Because of the form of $\psi_k(\xi,\eta)$, it is sufficient to bound the integral
$$J(x,y,t):=\int_{\mathbb R^2} e^{i[t(\xi^3+\eta^3)+x\xi+y\eta]}\phi(\xi)\phi(\eta) d\xi d\eta,$$
where $\phi(\nu):=\mu(2^{k+1}-|\nu|)$.

Let $\{\rho_1,\rho_2\}$ be a partition of unity of $\mathbb R$ subordinated to the open sets $\{\xi:|\xi|>1\}$ and $\{\xi: |\xi|<2\}$, respectively. Then
$$J(x,y,t)=\sum_{j=1}^2 \int_{\mathbb R^2} e^{i[t(\xi^3+\eta^3)+x\xi+y\eta]} \phi(\xi) \rho_j(\xi) \phi (\eta) d\xi d\eta \equiv J_1(x,y,t)+J_2(x,y,t).$$
For $j=1,2$, let $\Phi_j(\xi)=\phi(\xi)\rho_j(\xi)$.

\textit{Estimation of $J_1(x,y,t)$.}

We consider two cases:

\begin{enumerate}
\item[i)] First case: $T>0$ and $k\in\mathbb N$ such that $48T 2^{2k}\leq 2^{-k/2}$.\\
If $|x|\leq 2^{-k/2}$ it is obvious that
\begin{align}
|J_1(x,y,t)|\leq C 2^{2k} \text{ for } y\in\mathbb R\text{ and } t\in[0,T].\label{mte5}
\end{align}
If $2^{-k/2}<|x|$ and $\xi\in supp\, \Phi_1$ ($1<|\xi|<2^{k+1}$), then for $t\in[0,T]$
\begin{align}
|\varphi'(\xi)|=|3t\xi^2+x|\geq\dfrac{|x|}2\geq 3t\xi^2.\label{mte6}
\end{align}
Integrating twice by parts with respect to $\xi$, it can be seen that for $t\in[0,T]$ and $y\in\mathbb R$,
\begin{align}
|J_1(x,y,t)|\leq \int_{\mathbb R} \phi(\eta) \int_{\mathbb R} \left| \left[ \dfrac 1{\varphi'(\xi)} \left( \dfrac{\Phi_1(\xi)}{\varphi'(\xi)} \right)' \right]' \right| d\xi d\eta.\label{mte7}
\end{align}
In order to bound the right hand side of inequality \eqref{mte7} we take into account that
\begin{align*}
\left[ \dfrac 1{\varphi'(\xi)} \left( \dfrac{\Phi_1(\xi)}{\varphi'(\xi)} \right)' \right]'&=\dfrac{\Phi_1''(\xi)}{(\varphi'(\xi))^2}-\dfrac{18 \Phi_1'(\xi) t\xi}{(\varphi'(\xi))^3}-\dfrac{\Phi_1(\xi) 6t}{(\varphi'(\xi))^3}+\dfrac{3\Phi_1(\xi)(6t\xi)^2}{(\varphi'(\xi))^4}\\
&\equiv I+II+III+IV.
\end{align*}
Since the length of the set $\{\xi:\Phi_1''(\xi)\neq 0\}$ is less than or equal to 4, from \eqref{mte6} it follows that
\begin{align}
\int_{\mathbb R} \phi(\eta) \int_\mathbb R |I| d\xi d\eta\leq \dfrac C{x^2} \int_{\mathbb R} \phi(\eta) d\eta\leq \dfrac {C 2^k}{x^2}.\label{mte8}
\end{align}
Since $\{\xi:\Phi_1'(\xi)\neq 0\}\subset \{\xi:1<|\xi|<2\}\cup \{\xi:2^{k+1}-1<|\xi|<2^{k+1}\}$, from \eqref{mte6} we have that
\begin{align}
\notag \int_{\mathbb R} \phi(\eta) \int_{\mathbb R} |II|d\xi d\eta &\leq \dfrac C{x^2} \int_{\mathbb R} \phi(\eta) \int_{\mathbb R} \dfrac{|\Phi_1'(\xi)| t|\xi|}{t\xi^2} d\xi d\eta\\
\notag&\leq \dfrac C{x^2} \int_{\mathbb R} \phi(\eta) \left[\int_1^2 \dfrac 1\xi d\xi+\int_{2^{k+1}-1}^{2^{k+1}} \dfrac 1\xi \right] d\eta\\
&\leq \dfrac C{x^2} \int_{\mathbb R} \phi(\eta) d\eta \leq \dfrac {C2^k}{x^2}.\label{mte9}
\end{align}
Since $supp\, \Phi_1(\xi)\subset \{\xi:1\leq |\xi|\leq 2^{k+1}\}$, then from \eqref{mte6} we conclude that
\begin{align}
\notag\int_{\mathbb R} \phi(\eta) \int_\mathbb R |III| d\xi d\eta &\leq \dfrac C{x^2} \int_{\mathbb R} \phi(\eta) \int_{\mathbb R} \dfrac{\Phi_1(\xi)t}{t\xi^2} d\xi d\eta\\
&\leq \dfrac C{x^2} \int_{\mathbb R} \phi(\eta) \int_1^{2^{k+1}} \dfrac 1{\xi^2} d\xi d\eta\leq \dfrac{C 2^k}{x^2}, \label{mte10}
\end{align}
and
\begin{align}
\notag\int_\mathbb R \phi(\eta) \int_{\mathbb R} |IV| d\xi d\eta&\leq \dfrac C{x^2} \int_\mathbb R \phi(\eta) \int_{\mathbb R} \dfrac{\Phi_1(\xi)t^2 \xi^2}{(t\xi^2)^2} d\xi d\eta\\
&\leq \dfrac C{x^2} \int_{\mathbb R^2} \phi(\eta) \int_1^{2^{k+1}} \dfrac 1{\xi^2} d\xi d\eta \leq \dfrac{C2^k}{x^2}.\label{mte11}
\end{align}
From \eqref{mte7} to \eqref{mte11} it follows that
\begin{align}
|J_1(x,y,t)|\leq \dfrac{C 2^k}{x^2}\text{ for } 2^{-k/2}<|x|,\; y\in\mathbb R \text{ and } t\in[0,T].\label{mte12}
\end{align}
Let us define
$$H_{k1}(x):=\left\{ \begin{array}{ccc} C2^{2k} & \text{if} & |x|\leq 2^{-k/2},\\
\dfrac{C2^k}{x^2} & \text{if}& |x|>2^{-k/2}.  \end{array}\right.$$
Then from \eqref{mte5} and \eqref{mte12} we can conclude that
\begin{align}
|J_1(x,y,t)|\leq H_{k1}(x)\text{ for } (x,y)\in\mathbb R^2, \text{ and } t\in[0,T],\label{mte13}
\end{align}
where
\begin{align}
\| H_{k1}\|_{L^1}\leq C 2^{2k} 2^{-k/2}+C 2^k \int_{2^{-k/2}}^\infty \dfrac 1{x^2} dx=C 2^{3k/2}.\label{mte14}
\end{align}

\item[ii)] Second case: $T>0$ and $k\in \mathbb N$ such that $2^{-k/2}<48T 2^{2k}$.\\
If $|x|\leq 2^{-k/2}$ it is clear that \eqref{mte5} holds. If $x>2^{-k/2}$, $\varphi'(\xi)=|\varphi'(\xi)|=3t\xi^2+x$, and in consequence
\begin{align}
\varphi'(\xi)>3t\xi^2\text{ and } \varphi'(\xi)\geq x.\label{mte15}
\end{align}
If $x<-48T 2^{2k}$ and $\xi\in supp\, \Phi_1$, then for $t\in[0,T]$
\begin{align}
|\varphi'(\xi)|>\dfrac{|x|}2\geq 3t\xi^2.\label{mte16}
\end{align}
From \eqref{mte15} and \eqref{mte16}, proceeding as it was done in the first case, we have that if $x<-48T 2^{2k}$ or $x>2^{-k/2}$, then
\begin{align}
|J_1(x,y,t)|\leq \dfrac{C2^k}{x^2} \text{ for } y\in\mathbb R\text{ and } t\in[0,T].\label{mte17}
\end{align}
Let us suppose that $-48T2^{2k}<x<-2^{-k/2}$ and $t\in[0,T]$. If $x\leq-48t 2^{2k}$ and $\xi\in supp\,\Phi_1$, then inequalities \eqref{mte16} and \eqref{mte17} hold. If $-48T 2^{2k}\leq -48t 2^{2k}<x$, i.e. $\dfrac 1t\leq \dfrac{48\cdot2^{2k}}{|x|}$ then
\begin{align}
\notag|J_1(x,y,t)|&=\left|\int_\mathbb R e^{iy\eta} e^{it\eta^3} \phi(\eta) d\eta \right| \left|\int_\mathbb R e^{i\varphi(\xi)} \Phi_1 (\xi) d\xi \right|\\
\notag&=C\left| [(\mathcal F^{-1}\phi)\ast \mathcal F^{-1}(e^{it\eta^3})](y)\right| \left| \int_\mathbb R e^{i\varphi(\xi)}\Phi_1(\xi) d\xi \right|\\
&=C\left|\int_\mathbb R (\mathcal F^{-1}\phi)(y-z) \dfrac 1{t^{1/3}} A_i\left(\dfrac z{(3t)^{1/3}} \right) dz \right| \left| \int_\mathbb R e^{i\varphi(\xi)} \Phi_1(\xi) d\xi \right|,\label{mte18}
\end{align}
where $A_i$ is the Airy function and
\begin{align}
\|\mathcal F^{-1} \phi \|_{L^1}\leq C(k+1).\label{mte19}
\end{align}
Let us split $\mathbb R$ in the sets $\Omega_1:=\{ \xi: \xi^2>\dfrac{|x|}{48t}\}$ and $\Omega_2:=\{ \xi: \xi^2\leq \dfrac{|x|}{48t}\}$. If $\xi\in\Omega_1$,
$$|\varphi''(\xi)|=6t|\xi|\geq Ct\dfrac{|x|^{1/2}}{t^{1/2}}=Ct^{1/2}|x|^{1/2}.$$
Hence, by the Vander Courput's lemma (see \cite{S1986}, pages 309-311), we have that
\begin{align}
\left| \int_{\Omega_1} e^{i\varphi(\xi)}\Phi_1(\xi) d\xi  \right|\leq C(t^{1/2}|x|^{1/2})^{-1/2}=Ct^{-1/4}|x|^{-1/4}. \label{mte20}
\end{align}
If $\xi\in\Omega_2$,
$$|\varphi'(\xi)|=|3t\xi^2+x|\geq|x|-3t\xi^2\geq |x|-\dfrac{|x|}{16}\geq \dfrac{|x|}2.$$
Then, integrating by parts with respect to $\xi$, we have:
\begin{align*}
\left| \int_{\Omega_2} e^{i\varphi(\xi)} \Phi_1(\xi) d\xi \right|&=\left| \int_{\Omega_2} e^{i\varphi(\xi)} i\varphi'(\xi) \dfrac{\Phi_1(\xi)}{i\varphi'(\xi)} d\xi  \right|\\
&=\left|\dfrac{\Phi_1(\xi)}{i\varphi'(\xi)} e^{i\varphi(\xi)} \Big|_{\xi\in\partial \Omega_2}- \int_{\Omega_2} e^{i\varphi(\xi)} \left[ \dfrac{\Phi_1(\xi)}{\varphi'(\xi)}\right]' d\xi \right|\\
&\leq \dfrac 2{|\varphi'(\xi)|} \Big|_{\xi=\frac{|x|^{1/2}}{|48t|^{1/2}}}+\int_{\Omega_2} \left| \left[\dfrac{\Phi_1(\xi)}{\varphi'(\xi)} \right]' \right| d\xi\\
&\leq \dfrac 2{\left| \dfrac 1{16} |x|+x \right|} +\int_{\Omega_2} \dfrac{|\Phi_1'(\xi)|}{|\varphi'(\xi)|} d\xi+\int_{\Omega_2} \dfrac{|\Phi_1(\xi)||\varphi''(\xi)|}{|\varphi'(\xi)|^2} d\xi.
\end{align*}
Since the length of the set $\{ \xi: \Phi_1'(\xi) \neq 0\}$ is less than or equal to 4, it follows that
$$\int_{\Omega_2} \dfrac{|\Phi_1'(\xi)|}{|\varphi'(\xi)|}\leq \dfrac{C}{|x|}.$$
On the other hand
$$\int_{\Omega_2} \dfrac{|\Phi_1(\xi)| |\varphi''(\xi)|}{(\varphi'(\xi))^2} d\xi \leq \dfrac C{x^2} \int_{\Omega_2} 6t |\xi| d\xi \leq \dfrac{Ct}{x^2} \int_0^{\frac{|x|^{1/2}} {(48t)^{1/2}}} \xi d\xi \leq \dfrac C{|x|}.$$
In consequence,
\begin{align}
\left| \int_{\Omega_2} e^{i\varphi(\xi)} \Phi_1(\xi) d\xi \right|\leq \dfrac C{|x|}.\label{mte21}
\end{align}
From \eqref{mte18} to \eqref{mte21}, taking into account that the Airy function is bounded, we conclude that,
\begin{align*}
|J_1(x,y,t)|&\leq C(k+1) t^{-1/3} (t^{-1/4}|x|^{-1/4}+|x|^{-1})\leq C(k+1)(t^{-7/12} |x|^{-1/4}+t^{-1/3}|x|^{-1}).
\end{align*}
Because of the fact that $\dfrac 1t\leq C\dfrac{2^{2k}}{|x|}$, we have
\begin{align}
\notag |J_1(x,y,t)|&\leq C(k+1) \left( \dfrac{2^{7k/6}}{|x|^{7/12}} |x|^{-1/4} +\dfrac{2^{2k/3}}{|x|^{1/3}} |x|^{-1} \right)\\
&\leq C(k+1) (2^{7k/6}|x|^{-5/6}+2^{2k/3}|x|^{-4/3}). \label{mte22}
\end{align}
Let us define
$$H_{k1}(x):=\left\{ \begin{array}{ccc} C2^{2k} & \text{if} & |x|\leq 2^{-k/2},\\
\dfrac{C2^k}{x^2} & \text{if}& x\leq -48T 2^{2k} \text{ or } x>2^{-k/2},\\
C\left( \dfrac{2^k}{x^2}+(k+1)(2^{7k/6}|x|^{-5/6}+2^{2k/3}|x|^{-4/3})\right) & \text{if} & -48T 2^{2k}<x<-2^{-k/2}. \end{array}\right.$$
Then, taking into account inequalities \eqref{mte5}, \eqref{mte17} and \eqref{mte22}, it follows that
\begin{align}
|J_1(x,y,t)|\leq H_{k1}(x)\text{ for } (x,y)\in\mathbb R^2 \text{ and } t\in[0,T], \label{mte23}
\end{align} 
and
\begin{align}
\|H_{k1} \|_{L^1}\leq C(1+T^{1/6})(k+1) (2^{3k/2}+2^{5k/6})\leq C(1+T^{1/6})(k+1)2^{3k/2}.\label{mte24}
\end{align}
From \eqref{mte13} - \eqref{mte14} and \eqref{mte23}-\eqref{mte24} we have that for $k\geq 1$, there exists $H_{k1}\in L^1(\mathbb R)$ such that for $t\in[0,T]$ and $(x,y)\in\mathbb R^2$, $|J_1(x,y,t)|\leq H_{k1}(x)$ and estimate \eqref{mte24} for $\|H_{k1} \|_{L^1}$ holds.
\end{enumerate}

\textit{Estimation of $J_2(x,y,t)$}.

Let $T>0$, $t\in[0,T]$ and $(x,y)\in\mathbb R^2$. If $|x|\leq 24T$ then
\begin{align}
|J_2(x,y,t)|\leq \text{Area} \{ (\xi,\eta): |\xi|<2, |\eta<2^{k+1}|\}\leq C 2^k.\label{mte25}
\end{align}
If $|x|>24T$, $t\in[0,T]$ and $\xi\in supp\,\Phi_2$ we have
\begin{align}
|\varphi'(\xi)|=|3t\xi^2+x|\geq |x|-3t\xi^2\geq |x|-12T>\dfrac{|x|}2>3t\xi^2.\label{mte26}
\end{align}
Integrating twice by parts with respect to $\xi$, and using \eqref{mte26}, it follows that
\begin{align}
|J_2(x,y,t)|\leq \int \phi(\eta) \int \left| \left[ \dfrac1{\varphi'(\xi)} \left( \dfrac{\Phi_2(\xi)}{\varphi'(\xi)}\right)'\right] \right| d\xi d\eta\leq C\dfrac{2^k}{x^2}.\label{mte27}
\end{align}
For $T>1$, let us define
$$H_{k2}(x):=\left\{ \begin{array}{ccc} C2^{2k} & \text{if} & |x|\leq 24T,\\
\dfrac{C2^k}{x^2} & \text{if}& |x|>24T, \end{array}\right.$$
and for $0<T\leq 1$, let us define
$$H_{k2}(x):=\left\{ \begin{array}{ccc} C2^{2k} & \text{if} & |x|\leq 24,\\
\dfrac{C2^k}{x^2} & \text{if}& |x|>24, \end{array}\right.$$
From \eqref{mte25} and \eqref{mte27} we have that
\begin{align}
|J_2(x,y,t)|\leq H_{k2}(x)\text{ for } (x,y)\in\mathbb R^2 \text{ and } t\in [0,T], \label{mte28}
\end{align}
and
\begin{align}
\|H_{k2} \|_{L^1}\leq C(1+T)2^k. \label{mte29}
\end{align}
From estimates \eqref{mte23} and \eqref{mte28} for $J_1$ and $J_2$, respectively, taking into account \eqref{mte24} and \eqref{mte29}, we conclude that there exists $H_k\in L^1(\mathbb R)$ such that for $(x,y)\in\mathbb R^2$ and $t\in[0,T]$,
\begin{align}
|J(x,y,t)|\leq H_k(x), \; \forall x\in\mathbb R \label{mte30}
\end{align}
and
\begin{align}
\|H_k \|_{L^1(\mathbb R)}\leq C (1+T)(k+1) 2^{3k/2}. \label{mte31}
\end{align}
Because of the form of $\psi_k(\xi,\eta)$, the assertions \eqref{mte30} and \eqref{mte31} also are true for $I_k(x,y,t)$ instead of $J(x,y,t)$.

We apply now the results obtained for the integrals $I_k(x,y,t)$ to estimate the group $V$. For $k=0,1,2,\dots$, let
$$[V_k(t)v_0](x,y):=\int_{\mathbb R^2} e^{i[t(\xi^3+\eta^3)+x\xi+y\eta]} \psi_k(\xi,\eta) \widehat v_0 (\xi,\eta) d\xi d\eta.$$

Then
$$[V(t)v_0](x,y)=\sum_{k=0}^\infty [V_k(t)v_{0_k}](x,y),$$

where $\widehat v_{0_k}(\xi,\eta):=\widehat v_0 (\xi,\eta) \chi_{_{supp\, \Psi_k}}(\xi,\eta)$. (Here $\chi_{_{supp\, \Psi_k}}$ is the characteristic function of the set $supp\, \psi_k$ in $\mathbb R^2$).

Therefore
\begin{align}
\| V(\cdot) v_0\|_{L^2_x L^\infty_{Ty}}\leq \sum_{k=0}^\infty \| V_k(\cdot) v_{0_k}\|_{L_x^2 L_{Ty}^\infty}. \label{mte32}
\end{align}

Using duality, an argument due to Tomas \cite{T1975}, and taking into account estimates \eqref{mte4} and \eqref{mte31} it can be proved that
$$\|V_k(\cdot)v_{0_k}\|_{L^2_x L^\infty _{Ty}}\leq C (1+T)^{1/2} (k+1)^{1/2}2^{3k/4} \|v_{0_k} \|_{L^2_{xy}},\quad k=0,1,2,3,\dots$$

Then, for $s>\dfrac 34$,
\begin{align}
 \|V(\cdot)v_0 \|_{L^2_x L^\infty _{yT}}\leq &C (1+T)^{1/2} \sum_{k=0}^\infty (k+1)^{1/2} 2^{(3/4-s)k} 2^{sk} \|v_{0_k} \|_{L^2_{xy}}.\label{mte33}
\end{align}

Let us observe that
\begin{align*}
2^{sk}\|v_{0_k} \|_{L^2_{xy}}=&2^{sk}\left( \int_{\mathbb R^2} |\widehat v_{0_k} (\xi,\eta)|^2 d\xi d\eta \right)^{1/2}= \left( \int_{\mathbb R^2} 2^{2sk}|\chi_{_{supp\, \Psi_k}}(\xi, \eta)|^2 |\widehat v_0(\xi,\eta)|^2 d\xi d\eta\right)^{1/2}.
\end{align*}

For $(\xi,\eta)\in supp\, \psi_k$, $2^k-1<|\xi|<2^{k+1}$ or $2^k-1<|\eta|<2^{k+1}$. In particular, $\dfrac 12 2^k<|\xi|<2\cdot 2^k$ or $\dfrac 12 2^k <|\eta|<2\cdot 2^k$. In this manner, for $(\xi,\eta)\in supp\, \psi_k$, $2^{2k}<4(\xi^2+\eta^2)$ and in consequence it follows that
\begin{align}
2^{sk}\| v_{0_k}\|_{L^2_{xy}}\leq  \left( \int_{\mathbb R^2} 4^s (\xi^2+\eta^2)^s |\widehat v_0 (\xi,\eta)|^2 d\xi d\eta \right)^{1/2}\leq C 2^s \| D^s v_0 \|_{L^2_{xy}}. \label{mte34}
\end{align}

From \eqref{mte33} and \eqref{mte34} we conclude that
\begin{align*}
\|V(\cdot)v_0\|_{L_x^2 L_{Ty}^\infty} \leq& C 2^s \left( \sum_{k=0}^\infty (k+1)^{1/2} 2^{(3/4-s)k} \right) (1+T)^{1/2} \|D^s v_0 \|_{L^2_{xy}}\leq C_s (1+T)^{1/2} \| D^s v_0\|_{L^2_{xy}},
\end{align*}
and Lemma 2.3 is proved.
\end{proof}

\subsection{Leibniz rule and interpolation lemma}

In this subsection we recall the Leibniz rule for fractional derivatives, obtained in \cite{KPV1993}, and an interpolation inequality, which was deduced in \cite{NP2009} and \cite{FP2011}.

\begin{lemma}\label{l2.4} (Leibniz rule). Let us consider $0<\alpha<1$ and $1<p<\infty$. Thus
\begin{align*}
\| D^\alpha(fg)-f D^\alpha g-gD^\alpha f \|_{L^p(\mathbb R)}\leq C \|g \|_{L^\infty(\mathbb R)} \| D^\alpha f\|_{L^p(\mathbb R)},
\end{align*}
where $D^\alpha$ denotes $D^\alpha_x$ or $D^\alpha_y$.
\end{lemma}

With respect to the weight $\langle r\rangle:= (1+r^2)^{1/2}$, for $N\in\mathbb N$, we will consider a truncated weight $ w_N$ of $\langle r \rangle$, such that $w_N\in C^\infty(\mathbb R)$,
\begin{align}
w_N(r):=\left\{ \begin{array}{ccc} (1+r^2)^{1/2} & \text{if}& |r|\leq N,\\ 2N & \text{if} & |r|\geq 3N,   \end{array} \right.\label{wsubn}
\end{align}

$w_N$ is non-decreasing in $|r|$ and, for $j\in\mathbb N$ and $r\in\mathbb R$,
$$|w_N^{(j)}(r)|\leq \dfrac{c_j}{w_N^{j-1}(r)},$$

where the constant $c_j$ is independent from $N$.

\begin{lemma}\label{interpol} (Interpolation lemma). Let $a,b>0$ and assume that
$$J^a f:=(1-\Delta)^{a/2} f\in L^2(\mathbb R^n) \text{ and } \langle |x|\rangle^b f \in L^2(\mathbb R^n),$$ 
where $x=(x_1,\cdots, x_n)$ and $|x|=\left(\sum_{i=1}^n x_i^2 \right)^{1/2}$. 

Then, for any $\theta\in (0,1)$,
\begin{align}
\| \langle |x|\rangle^{\theta b} J^{(1-\theta )a } f \|_{L^2}\leq C \| \langle |x| \rangle^b f\|_{L^2}^\theta \| J^a f\|^{1-\theta}_{L^2}. \label{leinterpol}
\end{align}
Moreover, the inequality \eqref{leinterpol} is still valid with $w_N(|x|)$ instead of $\langle |x| \rangle$ with a constant $C$ independent of $N$.
\end{lemma}

\subsection{Stein derivative}

In this subsection, we obtain in Lemma \ref{lemastein1} an appropriate bound for $\mathcal D^b(e^{it\xi^3})(\xi,\eta)$. Then, using properties \eqref{intro5} and \eqref{intro6} of the Stein derivative and Lemma \ref{lemastein1}, we succeed, in Corollary \ref{c2.7}, to bound in  an adequate manner the weighted $L^2$-norm $\|(|x|+|y|)^b V(t) \|_{L^2_{xy}}$, for the group ot the symmetrized ZK equation.

\begin{lemma}\label{lemastein1} Let $b\in (0,1/2]$. For any $t>0$,
\begin{align*}
\mathcal D^b(e^{itx_1^3})(x_1,x_2)\leq C_b \left( t^{b/3} +t^{(b+1)/3} +t^{b/3} |x_1|^b +(t^{1/3+2b/3}+t^{2b/3}) |x_1|^{2b}\right).
\end{align*}
\end{lemma}

\begin{proof}
Let $x:=(x_1,x_2)$ and $y:=(y_1,y_2)$. After the change of variables $w:=t^{1/3}(x-y)$ we have that
\begin{align}
\notag \mathcal D^b(e^{itx_1^3})(x_1,x_2)&=\left(\int_{\mathbb R^2} \dfrac{|e^{itx_1^3}-e^{ity_1^3}|^2}{|x-y|^{2+2b}}dy \right)^{1/2}\\
&=t^{b/3}\left(\int_{\mathbb R^2} \dfrac{|e^{i(-3x_1^2 t^{2/3}w_1+3x_1t^{1/3}w_1^2-w_1^3)}-1|^2}{|w|^{2+2b}}dw \right)^{1/2}\equiv t^{b/3} I. \label{st0}
\end{align}

Let us observe that
$$|i(-3x_1^2t^{2/3}w_1+3x_1t^{1/3}w_1^2-w_1^3)|\leq |w_1|(3x_1^2 t^{2/3}+3|x_1|t^{1/3}|w_1|+w_1^2).$$

In consequence, for $w_1$ such that $3x_1^2 t^{2/3}>3|x_1|t^{1/3}|w_1|$, i.e. for $w_1$ such that $|x_1|t^{1/3}>|w_1|$, it follows that
\begin{align}
\notag |-3x_1^2 t^{2/3}w_1+3x_1 t^{1/3}w_1^2-w_1^3|&\leq |w_1| (6x_1^2t^{2/3}+|w_1|^2)\\
&\leq |w_1|(6x_1^2 t^{2/3}+x_1^2 t^{2/3})\leq 7x_1^2 t^{2/3}|w_1|.\label{st1}
\end{align}

In order to estimate $I$ we split the $\mathbb R^2$-plane in three regions $E_i$, $i=1,2,3$.

First, we define
$$E_2:=\{w=(w_1,w_2):|w_1|<t^{1/3}|x_1|,\; |w_1|<(t^{1/3}x_1^2)^{-1} \},$$

and we estimate 
$$\left(\int_{E_2} \dfrac{|e^{i(-3x_1^2 t^{2/3}w_1+3x_1t^{1/3}w_1^2-w_1^3)}-1|^2}{|w|^{2+2b}}dw \right)^{1/2}.$$

Two cases will be consider to estimate this integral.

\textit{Case 2.1.} $t^{1/3}|x_1|\leq t^{-1/3} x_1^{-2}$.

In this case, taking into account \eqref{st1}, we have
\begin{align}
\notag &\left(\int_{E_2} \dfrac{|e^{i(-3x_1^2 t^{2/3}w_1+3x_1t^{1/3}w_1^2-w_1^3)}-1|^2}{|w|^{2+2b}}dw \right)^{1/2}\\
\notag &\hspace{2cm}\leq C x_1^2 t^{2/3} \left( \int_0^{t^{1/3}|x_1|} \int_\mathbb R \dfrac{w_1^2}{(w_1^2+w_2^2)^{1+b}}dw_2 dw_1 \right)^{1/2}\\
\notag &\hspace{2cm}\leq C x_1^2 t^{2/3} \left( \int_0^{t^{1/3}|x_1|} \left[ \int_0^{w_1} \dfrac{w_1^2}{ w_1^{2+2b}}dw_2 + \int_{w_1}^\infty \dfrac{w_1^2}{w_2^{2+2b}}dw_2 \right] dw_1 \right)^{1/2}\\
&\hspace{2cm}\leq C x_1^2 t^{2/3} \left( \int_0^{t^{1/3}|x_1|} (w_1^{1-2b}+ w_1^{1-2b}) dw_1 \right)^{1/2}\leq C x_1^2 t^{2/3} (t^{1/3} |x_1|)^{1-b}\leq C t^{1/3-b/9},\label{st1b}
\end{align}
where in the last inequality the condition $|x_1|^3< t^{-2/3}$ was used.

\textit{Case 2.2.} $t^{1/3}|x_1|> t^{-1/3} x_1^{-2}$.

A simple calculation shows that
\begin{align}
\notag \left(\int_{E_2} \dfrac{|e^{i(-3x_1^2 t^{2/3}w_1+3x_1t^{1/3}w_1^2-w_1^3)}-1|^2}{|w|^{2+2b}}dw \right)^{1/2}&\leq C x_1^2 t^{2/3} \left(\int_0^{t^{-1/3}x_1^{-2}} w_1^{1-2b} dw_1 \right)^{1/2}\\
&\leq C t^{1/3+b/3} |x_1|^{2b}. \label{st1c}
\end{align}

From \eqref{st1b} and \eqref{st1c} we have that
\begin{align}
\left(\int_{E_2} \dfrac{|e^{i(-3x_1^2 t^{2/3}w_1+3x_1t^{1/3}w_1^2-w_1^3)}-1|^2}{|w|^{2+2b}}dw \right)^{1/2}\leq C (t^{1/3-b/9}+t^{1/3+b/3}|x_1|^{2b}).\label{st2}
\end{align}

For the region
$$E_1:=\{ w=(w_1,w_2): |w_1|> (t^{1/3} x_1^2)^{-1}\},$$

one has
\begin{align}
\notag&\left(\int_{E_1} \dfrac{|e^{i(-3x_1^2 t^{2/3}w_1+3x_1t^{1/3}w_1^2-w_1^3)}-1|^2}{|w|^{2+2b}}dw \right)^{1/2}\leq 2\left(\int_{E_1} \dfrac 1{|w|^{2+2b}} dw \right)^{1/2}\\
\notag&\leq C \left[\int_0^{(t^{1/3}x_1^2)^{-1}} \int_{(t^{1/3}x_1^2)^{-1}}^\infty \dfrac 1{w_1^{2+2b}} dw_1 dw_2 +\int_{(t^{1/3}x_1^2)^{-1}}^\infty \int_{(t^{1/3}x_1^2)^{-1}}^{w_1}\dfrac 1{w_1^{2+2b}} dw_2 dw_1\right]^{1/2}\\
&\leq C_b \left[ (t^{1/3} x_1^2)^{-1}(t^{1/3}x_1^2)^{1+2b}+ (t^{1/3}x_1^2)^{2b}\right]^{1/2}\leq C_b t^{b/3}|x_1|^{2b}.\label{st3}
\end{align}

From \eqref{st2} and \eqref{st3}, if $\min\{t^{1/3} |x_1|,(t^{1/3}x_1^2)^{-1}\}=(t^{1/3}x_1^2)^{-1}$, we obtain that
\begin{align}
\left(\int_{\mathbb R^2} \dfrac{|e^{i(-3x_1^2 t^{2/3}w_1+3x_1t^{1/3}w_1^2-w_1^3)}-1|^2}{|w|^{2+2b}}dw \right)^{1/2}\leq C_b [t^{1/3-b/9}+(t^{1/3+b/3}+t^{b/3})|x_1|^{2b}].\label{st4}
\end{align}

Now we consider the case $\min\{t^{1/3} |x_1|,(t^{1/3}x_1^2)^{-1}\}=t^{1/3}|x_1|$, i.e. $|x_1|^3 t^{2/3}<1$, and for that purpose we define
$$E_3:=\{ w=(w_1,w_2): t^{1/3} |x_1|<|w_1|< (t^{1/3} x_1^2)^{-1}\}.$$

In order to estimate
$$\left(\int_{E_3} \dfrac{|e^{i(-3x_1^2 t^{2/3}w_1+3x_1t^{1/3}w_1^2-w_1^3)}-1|^2}{|w|^{2+2b}}dw \right)^{1/2},$$
we need to consider three cases.

\textit{Case 3.1.} $1<t^{1/3}|x_1|$.

For this case we note that
\begin{align}
\notag& \hspace{-0.5cm}\left(\int_{E_3} \dfrac{|e^{i(-3x_1^2 t^{2/3}w_1+3x_1t^{1/3}w_1^2-w_1^3)}-1|^2}{|w|^{2+2b}}dw \right)^{1/2}\\
\notag \leq &C \left( \int_{E_3} \dfrac 1{|w|^{2+2b}} \right)^{1/2}\leq C\left(\int_0^\infty \int_{t^{1/3}|x_1|}^{(t^{1/3}x_1^2)^{-1}} \dfrac 1{(w_1^2+w_2^2)^{1+b}} dw_1 dw_2\right)^{1/2}\\
\notag\leq & C \left[ \int_0^{(t^{1/3}x_1^2)^{-1}} \left( \int_{t^{1/3}|x_1|}^{(t^{1/3}x_1^2)^{-1}} \dfrac 1{w_1^{2+2b}} dw_1\right)dw_2  +\int_{(t^{1/3}x_1^2)^{-1}}^\infty \left( \int_{t^{1/3}|x_1|}^{(t^{1/3}x_1^2)^{-1}} \dfrac 1{w_2^{2+2b}} dw_1 \right) dw_2 \right]^{1/2}\\
\equiv& C(I_{31}+I_{32})^{1/2}.\label{st4a}
\end{align}

It is easy to check that
\begin{align}
(I_{31})^{1/2}\leq C \left(\dfrac 1{(t^{1/3}|x_1|)^{2+2b}} [(t^{1/3}x_1^2)^{-1}]^2 \right)^{1/2}=\dfrac{C(t^{1/3}x_1^2)^{-1}}{(t^{1/3}|x_1|)^{1+b}}=Ct^{1/3}t^{-(3+b)/3} |x_1|^{-(3+b)}\leq Ct^{1/3},\label{st5}
\end{align}
where in the last inequality we use the condition $t^{-1/3}|x_1|^{-1}<1$.

Besides,
\begin{align}
(I_{32})^{1/2}\leq  \left( \int_{(t^{1/3}x_1^2)^{-1}}^\infty \dfrac 1{w_2^{2+2b}} (t^{1/3}x_1^2)^{-1} dw_2\right)^{1/2}\leq C t^{b/3} |x_1|^{2b}.\label{st7}
\end{align}

Hence, from \eqref{st4a} to \eqref{st7} we conclude that
\begin{align}
\left(\int_{E_3} \dfrac{|e^{i(-3x_1^2 t^{2/3}w_1+3x_1t^{1/3}w_1^2-w_1^3)}-1|^2}{|w|^{2+2b}}dw \right)^{1/2}\leq C(t^{1/3}+t^{b/3}|x_1|^{2b}).\label{st8}
\end{align}

\textit{Case 3.2.} $t^{1/3}|x_1|<1<(t^{1/3}x_1^2)^{-1}$.

Let us observe that for $|w_1|<1$,
\begin{align}
|w_1(-3x_1^2t^{2/3}+3x_1t^{1/3}w_1-w_1^2)|\leq |w_1| (3+3|w_1|+|w_1|^2)\leq C|w_1|,\label{st8as}
\end{align}

and then
\begin{align}
\notag&\left(\int_{E_3} \dfrac{|e^{i(-3x_1^2 t^{2/3}w_1+3x_1t^{1/3}w_1^2-w_1^3)}-1|^2}{|w|^{2+2b}}dw \right)^{1/2}\\
&\leq C \left( \int_0^\infty \int_{t^{1/3}|x_1|}^1 \dfrac{|w_1|^2}{|w|^{2+2b}} dw_1 dw_2 +  \int_0^\infty \int_1^{(t^{1/3}x_1^2)^{-1}} \dfrac 1{|w|^{2+2b}} dw_1 dw_2 \right)^{1/2}\equiv C (II_{31}+II_{32})^{1/2}.\label{st8a}
\end{align}

For $(II_{31})^{1/2}$ we have
\begin{align}
\notag (II_{31})^{1/2}\leq &C\left( \int_0^{t^{1/3} |x_1|} \int_{t^{1/3}|x_1|}^1 \dfrac{w_1^2}{w_1^{2+2b}}dw_1dw_2 + \int_{t^{1/3}|x_1|}^1 \int_{t^{1/3}|x_1|}^1 \dfrac{w_1^2}{|w|^{2+2b}}dw_1dw_2\right. \\
& \left.+ \int_1^\infty \int_{t^{1/3}|x_1|}^1 \dfrac{w_1^2}{w_2^{2+2b}}dw_1dw_2\right)^{1/2}\equiv C(II_{311}+II_{312}+II_{313})^{1/2}.\label{st8b}
\end{align}

Let us estimate $(II_{311})^{1/2}$.

For $b\in(0,1/2]$, it follows that
\begin{align}
(II_{311})^{1/2}\leq \left(\int_0^{t^{1/3}|x_1|} \int_{t^{1/3}|x_1|}^1 \dfrac 1{w_1} dw_1 dw_2 \right)^{1/2}=\left[t^{1/3}|x_1|\ln(t^{1/3}|x_1|)^{-1}\right]^{1/2}\leq C.\label{st8c}
\end{align}

Now we estimate $(II_{312})^{1/2}$.

Taking into account that $b\in(0,1/2]$ we conclude that
\begin{align}
\notag (II_{312})^{1/2}\leq &\left( \int_{t^{1/3}|x_1|}^1 \int_{t^{1/3}|x_1|}^{w_1} \dfrac{w_1^2}{w_1^{2+2b}}dw_2 dw_1 +\int_{t^{1/3}|x_1|}^1 \int_{w_1}^1 \dfrac{w_1^2}{w_2^{2+2b}} dw_2 dw_1\right)^{1/2}\\
\leq & \left( 1+ \dfrac 1{1+2b} \right)^{1/2}=C.\label{st9}
\end{align}

And for $(II_{313})^{1/2}$ it is clear that
\begin{align}
(II_{313})^{1/2}\leq C \left( \int_1^\infty \dfrac 1{w_2^{2+2b}}dw_2 \right)^{1/2}\leq C.\label{st10}
\end{align}

From \eqref{st8b} to \eqref{st10} we have that
\begin{align}
(II_{31})^{1/2}\leq C.\label{st11}
\end{align}

The estimation of $(II_{32})^{1/2}$ is as follows:
\begin{align}
\notag (II_{32})^{1/2}\leq &\left(\int_0^1 \int_1^{(t^{1/3}x_1^2)^{-1}} \dfrac 1{w_1^{2+2b}} dw_1  dw_2 +\int_1^{(t^{1/3}x_1^2)^{-1}} \int_1^{(t^{1/3}x_1^2)^{-1}} \dfrac 1{|w|^{2+2b}} dw_1 dw_2\right.\\
\notag&+\left. \int_{(t^{1/3}x_1^2)^{-1}}^\infty \int_1^{(t^{1/3}x_1^2)^{-1}} \dfrac 1{w_2^{2+2b}} dw_1 dw_2\right)^{1/2}\\
\leq &(C+C_b+t^{2b/3} |x_1|^{4b} )^{1/2}\leq C_b+ t^{b/3} |x_1|^{2b}.\label{st12}
\end{align}

From \eqref{st11} and \eqref{st12} we can affirm that, for $b\in(0,1/2]$,
\begin{align}
\left(\int_{E_3} \dfrac{|e^{i(-3x_1^2 t^{2/3}w_1+3x_1t^{1/3}w_1^2-w_1^3)}-1|^2}{|w|^{2+2b}}dw \right)^{1/2}\leq C_b+t^{b/3}|x_1|^{2b}.\label{st12r}
\end{align}

\textit{Case 3.3.} $(t^{1/3}x_1^2)^{-1}<1$.

In this final case we obtain, using \eqref{st8as},
\begin{align}
\notag \left(\int_{E_3} \dfrac{|e^{i(-3x_1^2 t^{2/3}w_1+3x_1t^{1/3}w_1^2-w_1^3)}-1|^2}{|w|^{2+2b}}dw \right)^{1/2}\leq &C \left(\int_0^\infty \int_{t^{1/3}|x_1|}^{(t^{1/3}x_1^2)^{-1}} \dfrac{w_1^2}{w_1^{2+2b}+w_2^{2+2b}} dw_1 dw_2 \right)^{1/2}\\
\notag\leq& C \left( \int_{0}^{t^{1/3}|x_1|} \int_{t^{1/3}|x_1|}^{(t^{1/3}x_1^2)^{-1}} \dfrac{w_1^2}{w_1^{2+2b}}dw_1 dw_2 \right.\\
\notag &+\int_{t^{1/3}|x_1|}^{(t^{1/3}x_1^2)^{-1}} \int_{t^{1/3}|x_1|}^{(t^{1/3}x_1^2)^{-1}} \dfrac{w_1^2}{w_1^{2+2b}+w_2^{2+2b}}dw_1 dw_2\\
\notag &\left.+\int_{(t^{1/3}x_1^2)^{-1}}^{\infty} \int_{t^{1/3}|x_1|}^{(t^{1/3}x_1^2)^{-1}} \dfrac{w_1^2}{w_2^{2+2b}}dw_1 dw_2\right)^{1/2}\\
\equiv&C( III_{31}+III_{32}+III_{33})^{1/2}.\label{st12a}
\end{align}

It is easily seen that
\begin{align}
(III_{31})^{1/2}\leq \left(\int_0^{t^{1/3}|x_1|}\int_{t^{1/3}|x_1|}^1 w_1^{-2b} dw_1 dw_2\right)^{1/2} \leq \left(\int_0^{t^{1/3}|x_1|}\int_{t^{1/3}|x_1|}^1 w_1^{-1} dw_1 dw_2 \right)^{1/2}\leq C, \label{st13}
\end{align}

and that
\begin{align*}
(III_{32})^{1/2}\leq \left(\int_{t^{1/3}|x_1|}^{(t^{1/3}x_1^2)^{-1}} \int_{t^{1/3}|x_1|}^{(t^{1/3}x_1^2)^{-1}} w_1^{-2b} dw_1 dw_2 \right)^{1/2}.
\end{align*}

For $b\in (0,1/2)$,
\begin{align*}
(III_{32})^{1/2}\leq C_b ((t^{1/3}x_1^2)^{-1})^{1-b}\leq C_b,
\end{align*}

and, for $b=1/2$,
\begin{align*}
(III_{32})^{1/2}\leq &\left( \int_{t^{1/3}|x_1|}^{(t^{1/3}x_1^2)^{-1}} \ln(t^{1/3}|x_1|)^{-1} dw_2 \right)^{1/2}\leq \left((t^{1/3}x_1^2)^{-1} (t^{1/3}|x_1|)^{-1} (t^{1/3}|x_1|) \ln (t^{1/3}|x_1|)^{-1}\right)^{1/2}\\
\leq & Ct^{-1/3}x_1^{-2}|x_1|^{1/2}\leq C|x_1|^{1/2}\leq C|x_1|^b.
\end{align*}

Therefore, for $b\in(0,1/2]$,
\begin{align}
(III_{32})^{1/2}\leq C_b+C|x_1|^b.\label{st14}
\end{align}

Finally, we estimate $(III_{33})^{1/2}$.

\begin{align}
(III_{33})^{1/2}\leq C \left( (t^{1/3}x_1^2)^{-3}(t^{1/3}x_1^2)^{1+2b} \right)^{1/2}=C((t^{1/3}x_1^2)^{-1})^{1-b}\leq C.\label{st15}
\end{align}

From \eqref{st12a}, \eqref{st13}, \eqref{st14} and \eqref{st15} we have that, for $b\in(0,1/2]$,
\begin{align}
\left(\int_{E_3} \dfrac{|e^{i(-3x_1^2 t^{2/3}w_1+3x_1t^{1/3}w_1^2-w_1^3)}-1|^2}{|w|^{2+2b}}dw \right)^{1/2}\leq C_b+C|x_1|^b.\label{st16}
\end{align}

Consequently, from \eqref{st8}, \eqref{st12r} and \eqref{st16}, in any case, for $b\in(0,1/2]$,
\begin{align}
\left(\int_{E_3} \dfrac{|e^{i(-3x_1^2 t^{2/3}w_1+3x_1t^{1/3}w_1^2-w_1^3)}-1|^2}{|w|^{2+2b}}dw \right)^{1/2}\leq C_b(1+t^{1/3}+|x_1|^b+t^{b/3}|x_1|^{2b}). \label{st17}
\end{align}

Summarizing estimates \eqref{st0}, \eqref{st4} and \eqref{st17} imply that, for $b\in(0,1/2]$,
\begin{align*}
\mathcal D^b(e^{itx_1^3})(x_1,x_2)\leq &C_b t^{b/3} \left( 1+t^{1/3-b/9}+t^{1/3}+|x_1|^b+(t^{1/3+b/3}+t^{b/3}|x_1|^{2b}) \right)\\
\leq & C_b \left( t^{b/3} +t^{(b+1)/3} +t^{b/3} |x_1|^b +(t^{1/3+2b/3}+t^{2b/3}) |x_1|^{2b}\right).
\end{align*}
Lemma \ref{lemastein1} is proved.

\end{proof}

\begin{corollary}\label{c2.7} Let $\{V(t)\}_{t\in\mathbb R}$  be the group defined by \eqref{group}. For $b\in (0,1/2]$, there exists $C_b>0$ such that for $t\geq 0$ and $f\in H^{2b}(\mathbb R^2)\cap L^2((|x|+|y|)^{2b}dx dy)\equiv Z_{2b,b}$,
\begin{align}
\notag\| &(|x|+|y|)^b V(t) f\|_{L^2_{xy}}\\
&\leq C_b [(1+t^{b/3}+t^{(b+1)/3}) \| f\|_{L^2_{xy}}+(t^{b/3}+t^{1/3+2b/3}+t^{2b/3}) \|D^{2b}f \|_{L^2_{xy}}+\| (|x|+|y|)^b f \|_{L^2_{xy}}].\label{st21}
\end{align}
\end{corollary}

\begin{proof}
Taking into account the definition of $D^b$ (see \eqref{intro4}), Plancherel's theorem, the properties \eqref{intro3}, \eqref{intro6} and \eqref{intro5} of the Stein derivative of $\mathcal D^b$, and Lemma \ref{lemastein1} and using the notation $\text{}^{\vee}$ for the inverse Fourier transform, we have:
\begin{align*}
\| (|x|+|y|)^b V(t)f \|_{L^2}=& \| (|x|+|y|)^b (e^{it\xi^3+it\eta^3} \hat f) ^\vee\|_{L^2_{xy}}\\
=&\| (|-x|+|-y|)^b (e^{it\xi^3+it\eta^3} \hat f) ^\wedge (-x,-y)\|_{L^2_{xy}}\\
\leq& C\| [D^b(e^{it\xi^3+it\eta^3} \hat f)]^\wedge (-x,-y) \|_{L^2_{xy}}=C\| [D^b(e^{it\xi^3+it\eta^3} \hat f)]^\vee (x,y) \|_{L^2_{xy}}\\
=&C\| D^b(e^{it\xi^3+it\eta^3} \hat f) \|_{L^2_{\xi\eta}}\leq C\left( \| e^{it\xi^3+it\eta^3} \hat f \|_{L^2_{\xi\eta}}+\| \mathcal D^b (e^{it\xi^3+it\eta^3} \hat f) \|_{L^2_{\xi\eta}} \right)\\
\leq& C\left( \|f \|_{L^2_{xy}}+\|\hat f \mathcal D^b (e^{it\xi^3+it\eta^3}) \|_{L^2_{\xi\eta}}+\|\mathcal D^b (\hat f) (e^{it\xi^3+it\eta^3}) \|_{L^2_{\xi\eta}} \right)\\
\leq& C \left( \|f \|_{L^2_{xy}}+\|\hat f ( \mathcal D^b (e^{it\xi^3})+\mathcal D^b (e^{it\eta^3}) )\|_{L^2_{\xi\eta}}+\|\mathcal D^b (\hat f) \|_{L^2_{\xi\eta}} \right)\\
\leq & C \left( \|f \|_{L^2_{xy}}+C_b \|\hat f (t^{b/3}+t^{(b+1)/3}+t^{b/3}(|\xi|^b+|\eta|^b) \right.\\
& \left.+(t^{1/3+2b/3}+t^{2b/3})(|\xi|^{2b}+|\eta|^{2b})) \|_{L^2_{\xi\eta}}+C (\| \hat f \|_{L^2_{\xi\eta}}+\| D^b \hat f \|_{L^2_{\xi\eta}}) \right)\\
\leq &C \| f\|_{L^2_{xy}}+C_b (t^{b/3}+t^{(b+1)/3}) \| f\|_{L^2_{xy}}+C_b t^{b/3} \| D^b f \|_{L^2_{xy}}\\
&+C_b (t^{1/3+2b/3}+t^{2b/3}) \| D^{2b} f\|_{L^2_{xy}}+C\| (|x|+|y|)^b f \|_{L^2_{xy}}\\
\leq& C_b \left[(1+t^{b/3}+t^{(b+1)/3}) \|f \|_{L^2_{xy}}+(t^{b/3}+t^{1/3+2b/3}+t^{2b/3}) \| D^{2b} f\|_{L^2_{xy}}\right.\\
&\left.+ \| (|x|+|y|)^b f \|_{L^2_{xy}} \right].
\end{align*}

\end{proof}

\section{Proof of the main theorem}

\begin{proof} Case $3/4<s\leq 1$.

Following Grünrock and Herr in \cite{GH2014} we perform a linear change of variables in order to symmetrize the equation.

Let
\begin{align}
x':=\mu x+\lambda y,\quad y':=\mu x-\lambda y, \quad t':=t\quad \text{and} \quad v(x',y',t'):=u(x,y,t), \label{mt3.0}
\end{align}

where $\mu=4^{-1/3}$ and $\lambda=\sqrt 3 \mu$. Then $u$ satisfies the Z-K equation iff $v$ satisfies the equation
\begin{align*}
\partial_{t'}v+(\partial_{x'}^3 v+\partial_{y'}^3 v)+\mu (v\partial_{x'} v+ v\partial_{y'}v)=0.
\end{align*}

On the other hand, if
\begin{align}
v_0(x',y'):=u_0(x,y),\label{mt3.0a}
\end{align}

it easily can be seen that $v_0\in Z_{s,s/2}$ iff $u_0\in Z_{s,s/2}$.

In this manner we may consider the IVP
\begin{align}
\left.\begin{array}{rl} \partial_t v +(\partial_x^3 v+\partial_y^3 v)+\mu (v\partial_x v+v\partial_y v) &\hspace{-0.3cm}=0,\\
v(x,y,0)&\hspace{-0.3cm}=v_0(x,y)\in Z_{s,s/2},
\end{array}\right\}\label{mt3.1}
\end{align}

instead of IVP \eqref{ZK}, and the integral operator
\begin{align}
\Psi(v)(t):=V(t)v_0-\mu\int_0^t V(t-t')(v\partial_x v+v\partial_y v)(t') dt',\label{mt3.2}
\end{align}

where $\{V(t) \}_{t\in\mathbb R}$ is the unitary group associated to the linear part of the equation in \eqref{mt3.1}, i.e.,
\begin{align}
[V(t)v_0](x,y)=\dfrac 1{2\pi}\int_{\mathbb R^2} e^{i(t\xi^3+t\eta^3+x\xi+y\eta)} \widehat v_0(\xi,\eta)d\xi d\eta. \label{mt3.3}
\end{align}

Proceeding as in \cite{LP2009}, let us define, for $T>0$, the metric space
\begin{align}
X_T:=\{v\in C([0,T];H^s):  \n v \n <\infty \}, \label{mt3.4}
\end{align}

where
\begin{align}
\notag \n v \n :=& \| v\|_{L_T^\infty H_{xy}^s}+ \|D_x^s v_x \|_{L_x^\infty L_{yT}^2}+\|D_y^s v_x \|_{L_x^\infty L_{yT}^2}+ \| v_x\|_{L^2_T L^\infty_{xy}}+\|v \|_{L^2_{x}L^\infty_{yT}}+\|D^s_x v_y \|_{L^\infty_y L^2_{xT}}\\
&+\| D^s_y v_y \|_{L^\infty_y L^2_{xT}}+\|v_y \|_{L^2_T L^\infty_{xy}}+\| v\|_{L^2_y L^\infty_{xT}}+\|v \|_{L^\infty_T L^2((|x|+|y|)^s dxdy)}\equiv \sum_{i=1}^{10} n_i(v).\label{mt3.5}
\end{align}
(When $s=1$ in \eqref{mt3.5} we change $D_x^s$ and $D_y^s$ by $\partial_x$ and $\partial_y$, respectively).

For $a>0$, let $X_T^a$ be the closed ball in $X_T$ defined by
\begin{align}
X_T^a:= \{ v\in X_T: \n v \n\leq a  \}. \label{mt3.6}
\end{align}

We will prove that there exist $T>0$ and $a>0$ such that the operator $\Psi:X_T^a\to X_T^a$ is a contraction.

First of all let us prove that for $v_0\in Z_{s,s/2}$, $V(\cdot) v_0\in X_T$. Indeed
\begin{align}
\| V(\cdot)v_0\|_{L_T^\infty H_{xy}^s}=\| v_0\|_{H^s}<\infty.\label{mt3.7}
\end{align}

Using local type estimate \eqref{pr2.9}, we have
\begin{align}
\| D_x^s \partial_x V(\cdot) v_0\|_{L_x^\infty L_{yT}^2}=\|\partial_x V(\cdot) D_x^s v_0 \|_{L_x^\infty L_{yT}^2}\leq C\| D_x^s v_0\|_{L^2_{xy}}\leq C \| v_0\|_{H^s}<\infty; \label{mt3.8}
\end{align}
and
\begin{align}
\| D_y^s \partial_x V(\cdot) v_0 \|_{L^\infty_x L^2_{yT}}\leq C \| D_y^s v_0\|_{L^2_{xy}}\leq C\| v_0\|_{H^s}<\infty. \label{mt3.9}
\end{align}

From the Strichartz-type estimate \eqref{pr2.7} it follows that, for $\varepsilon \in (0,1/2]$,
\begin{align}
\notag \| \partial_x V(\cdot) v_0\|_{L^2_T L^\infty_{xy}}=&\| V(\cdot) \partial_x v_0\|_{L^2_T L^\infty_{xy}}\leq CT^\gamma \| D_x^{-\varepsilon/2} \partial_x v_0 \|_{L^2_{xy}}\leq CT^\gamma \|D_x^{(1-\varepsilon/2)}v_0 \|_{L^2_{xy}},
\end{align}
with $\gamma=\dfrac{1-\varepsilon}6$.

Since $s>3/4$, taking $\varepsilon\in (0,1/2)$ such that $s>1-\varepsilon/2$, from the last inequality we obtain
\begin{align}
\| \partial_x V(\cdot) v_0\|_{L^2_T L^\infty_{xy}}\leq C T^\gamma \| v_0\|_{H^s}<\infty. \label{mt3.10}
\end{align}

The maximal type estimate \eqref{pr2.11} implies
\begin{align}
\| V(\cdot)v_0\|_{L^2_x L^\infty_{yT}}\leq C_s (1+T)^{1/2} \| D^s v_0\|_{L^2_{xy}}\leq C_s(1+T)^{1/2} \| v_0\|_{H^s_{xy}}<\infty. \label{mt3.11}
\end{align}

On the other hand, using local type estimate \eqref{pr2.10}, Strichartz type estimate \eqref{pr2.8} and maximal type estimate \eqref{pr2.12}, we obtain
\begin{align}
\| D_x^s \partial_y V(\cdot)v_0\|_{L_y^\infty L_{xT}^2} \leq & C \|D_x^s v_0 \|_{L^2_{xy}}\leq C \| v_0\|_{H^s}<\infty; \label{mt3.12}\\
\| D_y^s \partial_y V(\cdot)v_0\|_{L_y^\infty L_{xT}^2}\leq & C \| D_y^s v_0\|_{L^2_{xy}}\leq C \| v_0\|_{H^s}<\infty; \label{mt3.13}\\
\| \partial_y V(\cdot) v_0 \|_{L^2_T L^\infty_{xy}}\leq &C T^\gamma \|v_0 \|_{H^s}<\infty; \label{mt3.14}\\
\| V(\cdot) v_0 \|_{L^2_y L^\infty_{xT}} \leq & C_s (1+T)^{1/2} \| D^s v_0\|_{L^2_{xy}}\leq C_s (1+T)^{1/2} \| v_0\|_{H^s}<\infty. \label{mt3.15}
\end{align}

Finally, from Corollary \ref{c2.7} in section 2.3, we have
\begin{align}
\notag\| V(\cdot) v_0&\|_{L_T^\infty L^2((|x|+|y|)^s dxdy)}=\sup_{t\in[0,T]} \left( \int_{\mathbb R^2} \left| [V(t)v_0](x,y) \right|^2(|x|+|y|)^s dxdy\right)^{1/2}\\
\notag=&\sup_{t\in[0,T]} \| (|x|+|y|)^{s/2} V(t)v_0 \|_{L^2_{xy}}\\
\notag \leq & C_s \left[ (1+T^{s/6}+T^{(s+2)/6}) \| v_0\|_{L^2_{xy}}+ (T^{s/6}+T^{1/3+s/3}+T^{s/3}) \| D^s v_0 \|_{L^2_{xy}}\right.\\
&\left.+\| (|x|+|y|)^{s/2} v_0\|_{L^2_{xy}}\right]<\infty. \label{mt3.16}
\end{align}

Estimates \eqref{mt3.7} to \eqref{mt3.16} imply that $V(\cdot)v_0 \in X_T$.

Let $v\in X_T^a$. We proceed to estimate $\n \Psi(v) \n$. For that it is necessary to bound all the norms $n_i$ that appear in the definition of $\n \cdot \n$.

\textit{Estimation of $n_1(\Psi(v))$.}

For $t\in[0,T]$, using \eqref{mt3.7}, we have
\begin{align}
\notag \| \Psi(v)(t)\|_{L^2_{xy}}\leq & \| V(t) v_0 \|_{L^2_{xy}}+C\int_0^T \| (v v_x)(t')\|_{L^2_{xy}} dt'+C\int_0^T \| (v v_y)(t') \|_{L^2_{xy}}dt'\\
\notag \leq & \|v_0 \|_{H^s}+C T^{1/2} (\|v v_x \|_{L^2_T L^2_{xy}} +\| v v_y\|_{L^2_T L^2_{xy}})\\
\notag \leq & \| v_0\|_{H^s} +CT^{1/2} (\| v \|_{L^\infty_T L^2_{xy}} \|v_x \|_{L^2_T L^\infty_{xy}}+\| v\|_{L^\infty_T L^2_{xy}} \| v_y\|_{L^2_T L^\infty_{xy}})\\
\leq & \| v_0 \|_{H^s}+C T^{1/2}\n v \n^2; \label{mt3.17}
\end{align}
and
\begin{align}
\| D_x^s \Psi (v) (t)\|_{L^2_{xy}}\leq \| D^s_x v_0\|_{L^2_{xy}}+C \int_0^T \| D_x^s (v v_x)(t')\|_{L^2_{xy}} dt' +C \int_0^T \| D_x^s (v v_y)(t') \|_{L^2_{xy}} dt'.\label{mt3.18}
\end{align}

We only estimate the first integral in \eqref{mt3.18}, being the estimation of the second one similar.

From Cauchy-Schwarz inequality and Leibniz rule for fractional derivatives (Lemma \ref{l2.4} in section 2.2) it follows that
\begin{align}
\notag&\int_0^T \|D_x^s(vv_x)(t') \|_{L^2_{xy}}dt'\leq  T^{1/2} \left( \int_0^T \hspace{-2mm}\int_{\mathbb R} \| D_x^s (vv_x)(t')(\cdot,y) \|^2_{L^2_x} dy dt' \right)^{1/2}\\
\notag \leq & CT^{1/2} \left( \int_0^T \hspace{-2mm}\int_{\mathbb R} \| v_x(t')(\cdot,y)\|^2_{L^\infty_x} \| D_x^s v(t')(\cdot,y)\|^2_{L^2_x} dy dt' + \int_0^T \hspace{-2mm} \int_{\mathbb R} \|v(t') D_x^s v_x(t')(\cdot,y) \|^2_{L^2_x} dy dt'\right)^{1/2}\\
\notag \leq& CT^{1/2} \left( \int_0^T \| v_x(t') \|^2_{L^\infty_{xy}} \| D_x^s v(t') \|^2_{L^2_{xy}} dt' + \int_0^T \| v(t') D_x^s v_x(t') \|_{L^2_{xy}}^2 dt' \right)^{1/2}\\
\notag \leq & CT^{1/2} \left( \| v\|^2_{L^\infty_T H^s_{xy}} \| v_x \|^2_{L^2_T L^\infty_{xy}} + \| v\|^2_{L^2_x L^\infty_{yT}} \|D^s_x v_x \|^2_{L^\infty_x L^2_{yT}}   \right)^{1/2}\\
\leq & CT^{1/2} \left( \|v \|_{L^\infty_T H^s_{xy}} \|v_x \|_{L^2_T L^\infty _{xy}}+ \|v \|_{L^2_x L^\infty_{yT}} \| D_x^s v_x \|_{L^\infty_x L^2_{yT}}\right)\label{mt3.19}.
\end{align}

From \eqref{mt3.18}, \eqref{mt3.19} and the similar estimation for the second integral in \eqref{mt3.18} we can conclude that
\begin{align}
\notag \| D_x^s \Psi(v)(t)\|_{L^2_{xy}} \leq & \| D_x^s v_0\|_{L^2_{xy}}+CT^{1/2} \left[\| v \|_{L^\infty_T H^s_{xy}}(\| v_x\|_{L^2_T L^\infty_{xy}}+\| v_y\|_{L^2_T L^\infty_{xy}})\right.\\
\notag &\left.+\| v\|_{L^2_x L^\infty_{yT}}\| D_x^s v_x\|_{L^\infty_x L^2_{yT}}+\| v\|_{L^2_y L^\infty_{xT}} \|D_x^s v_y \|_{L^\infty_y L^2_{xT}}\right]\\
\leq& \|D_x^s v_0 \|_{L^2_{xy}}+CT^{1/2} \n v\n^2.\label{mt3.20}
\end{align}

Similarly, it can be established that, for all $t\in[0,T]$,
\begin{align}
\|D_y^s \Psi(v)(t) \|_{L^2_{xy}}\leq \| D_y^s v_0\|_{L^2_{xy}}+C T^{1/2}\n v\n^2. \label{mt3.21}
\end{align}

Estimates \eqref{mt3.17}, \eqref{mt3.20} and \eqref{mt3.21} imply that
\begin{align}
\| \Psi(v) \|_{L^\infty_T H^s_{xy}}\leq C \| v_0\|_{H^s} +CT^{1/2} \n v \n^2.\label{mt3.22}
\end{align}

\textit{Estimation of $n_i(\Psi(v))$, $i=2,3,4$.} 
\begin{align*}
n_2(\Psi(v))\leq & \| D_x^s \partial_x (V(\cdot) v_0)\|_{L^\infty_x L^2_{yT}}+C \int_0^T \| D_x^s \partial_x V(\cdot_t) (V(-t') (v v_x)(t'))\|_{L^\infty _x L^2_{yT}} dt'\\
&+C\int_0^T \|D_x^s \partial_x V(\cdot_t)(V(-t')(vv_y)(t')) \|_{L^\infty_x L^2_{yT}} dt'.
\end{align*}

Taking into account \eqref{mt3.8}, estimate \eqref{mt3.19} for $D_x^s(vv_x)$ and similar one for $D_x^s(vv_y)$, it follows that
\begin{align}
\notag n_2(\Psi(v))\leq & C \| v_0\|_{H^s}+C\int_0^T \| D_x^s (v v_x)(t')\|_{L^2_{xy}} dt'+C\int_0^T \|D_x^s (vv_y)(t') \|_{L^2_{xy}} dt'\\
\leq& C\|v_0 \|_{H^s}+CT^{1/2} \n v\n ^2. \label{mt3.23}
\end{align}
From estimate \eqref{mt3.9}, proceeding in a similar manner as it was done in the estimation of $n_2(\Psi(v))$, it easily follows that
\begin{align}
n_3(\Psi(v))\equiv\| D_y^s \partial_x (\Psi(v))\|_{L^\infty_x L^2_{yT}}\leq C \| v_0\|_{H^s}+CT^{1/2} \n v\n^2.\label{mt3.24}
\end{align}

Estimate \eqref{mt3.10}, Cauchy-Schwarz inequality and Leibniz rule imply that
\begin{align}
n_4(\Psi(v))\equiv \| \partial_x (\Psi(v))\|_{L_T^2 L^\infty_{xy}}\leq C T^\gamma \|v_0 \|_{H^s}+CT^{\gamma+1/2} \n v \n^2.\label{mt3.25}
\end{align}

\textit{Estimation of $n_i(\Psi(v))$, $i=5,\dots,9$.}

Using \eqref{mt3.11} to \eqref{mt3.15}, and proceeding in a similar manner as it was done in the estimation of $n_i(\Psi(v))$, $i=1,\dots, 4$, we obtain
\begin{align}
n_5(\Psi(v))\equiv&\| \Psi(v)\|_{L^2_x L^\infty_{yT}}\leq C_s (1+T)^{1/2}\| v_0\|_{H^s}+C_s(1+T)^{1/2} T^{1/2} \n v \n^2, \label{mt3.26}\\
n_6(\Psi(v))\equiv&\| D_x^s \partial_y (\Psi(v))\|_{L^\infty_y L^2_{xT}}\leq  C \|v_0 \|_{H^s}+CT^{1/2} \n v \n^2, \label{mt3.27}\\
n_7(\Psi(v))\equiv&\| D^s_y \partial_y (\Psi(v))\|_{L^\infty_y L^2_{xT}} \leq  C \| v_0\|_{H^s}+CT^{1/2} \n v\n^2, \label{mt3.28}\\
n_8(\Psi(v))\equiv&\| \partial_y (\Psi(v))\|_{L^2_T L^\infty_{xy}}\leq  C T^\gamma \|v_0 \|_{H^s}+CT^{\gamma+1/2} \n v\n^2,\label{mt3.29}\\
n_9(\Psi(v))\equiv&\| \Psi(v)\|_{L^2_y L^\infty_{xT}}\leq  C_s (1+T)^{1/2} \| v_0\|_{H^s} +C_s (1+T)^{1/2}T^{1/2} \n v \n^2.\label{mt3.30}
\end{align}

\textit{Estimation of $n_{10}(\Psi(v))$.}

Applying Corollary \ref{c2.7} in section 2.3 we have, for $t\in[0,T]$, that
\begin{align*}
n_{10}(\Psi(v))\equiv&\|\Psi(v)(t) \|_{L^2((|x|+|y|)^s dxdy)}\\
\leq& \| V(t)v_0\|_{L^2((|x|+|y|)^sdxdy)}+C\|\int_0^t V(t-t')((vv_x)(t')+(vv_y)(t'))dt' \|_{L^2((|x|+|y|)^sdxdy)}\\
\leq& C_s \left[(1+t^{s/6}+t^{(s+2)/6})\| v_0\|_{L^2_{xy}}+(t^{s/6}+t^{1/3+s/3}+t^{s/3}) \|D^s v_0 \|_{L^2_{xy}}\right.\\
&\left.+ \| (|x|+|y|)^{s/2} v_0\|_{L^2_{xy}}\right]+C\int_0^t  C_s \left[(1+(t-t')^{s/6}+(t-t')^{(s+2)/6})(\| (v v_x)(t')\|_{L^2_{xy}}\right.\\
&+\| (vv_y)(t')\|_{L^2_{xy}})+((t-t')^{s/6}+(t-t')^{1/3+s/3}+(t-t')^{s/3}) ( \| D^s(vv_x)(t')\|_{L^2_{xy}}\\
&\left.+\|D^s(vv_y)(t') \|_{L^2_{xy}})+\|(|x|+|y|)^{s/2}((vv_x)(t')+(vv_y)(t')) \|_{L^2_{xy}}\right] dt'\\
\leq&C_s\left[(1+T^{1/3+s/3}) \|v_0 \|_{H^s}+\|(|x|+|y|)^{s/2}v_0\|_{L^2_{xy}}\right]\\
&+C_s(1+T^{1/3+s/3}) \int_0^T (\| (vv_x)(t')\|_{L^2_{xy}}+\| (vv_y)(t')\|_{L^2_{xy}}) dt'\\
&+C_s(1+T^{1/3+s/3}) \int_0^T (\| D^s(vv_x)(t')\|_{L^2_{xy}}+\| D^s(vv_y)(t')\|_{L^2_{xy}})dt'\\
&+C_sT^{1/2} (\| (|x|+|y|)^{s/2}(vv_x)\|_{L^2_T L^2_{xy}}+\| (|x|+|y|)^{s/2}(vv_y)\|_{L^2_T L^2_{xy}}).
\end{align*}

Taking into account that $1/3<4s/9$, it follows, for $t\in[0,T]$, that
\begin{align}
\notag n_{10}(\Psi(v)) \leq &C_s\left[(1+T^{7s/9}) \| v_0\|_{H^s}+\| (|x|+|y|)^{s/2}v_0 \|_{L^2_{xy}}\right]\\
\notag &+C_s (1+T^{7s/9}) T^{1/2}\n v\n^2 + C_sT^{1/2} (\| (|x|+|y|)^{s/2}v\|_{L_T^\infty L^2_{xy}} \| v_x\|_{L^2_T L^\infty_{xy}}\\
\notag &+\|(|x|+|y|)^{s/2}v \|_{L^\infty_T L^2_{xy}} \| v_y\|_{L^2_T L^\infty_{xy}})\\
\leq& C_s \left[(1+T^{7s/9}) \|v_0 \|_{H^s}+ \| (|x|+|y|)^{s/2} v_0\|_{L^2_{xy}}\right]+C_s(1+T^{7s/9})T^{1/2} \n v \n^2. \label{mt3.31}
\end{align}

From estimates \eqref{mt3.22} to \eqref{mt3.31}, taking into account that  $7s/9>1/2>\gamma$, we obtain
\begin{align}
\n \Psi(v)\n \leq C_s \left[(1+T^{7s/9}) \|v_0 \|_{H^s}+\|(|x|+|y|)^{s/2}v_0 \|_{L^2_{xy}}\right]+C_s(1+T^{7s/9})T^{1/2} \n v \n^2.\label{mt3.32}
\end{align}

If we choose
$$a:= 2C_s\left[(1+T^{7s/9}) \| v_0\|_{H^s}+\| (|x|+|y|)^{s/2} v_0 \|_{L^2_{xy}}\right],$$
and $T>0$ such that
$$C_s(1+T^{7s/9})T^{1/2} a<1/2,$$
it can be seen that $\Psi$ maps $X^a_T$ into itself. Moreover, for $T$ small enough, $\Psi:X_T^a\to X_T^a$ is a contraction. In consequence, there exists a unique $v\in X_T^a$ such that $\Psi(v)=v$. In other words, for $t\in [0,T]$,
$$v(t)=V(t)v_0-\mu \int_0^t V(t-t')(v\partial_x v +v \partial_y v)(t') dt',$$
i.e., the IVP \eqref{mt3.1} has a unique solution in $X_T^a$.

Using standard arguments, it is possible to show that for any $T'\in(0,T)$ there exists a neighborhood $W$ of $v_0$ in $Z_{s,s/2}$ such that the map $\tilde v_0\to \tilde v$ from $W$ into the metric space $X_{T'}$, with $T'$ instead of $T$, is Lipschitz. Then the assertion of Theorem 1.1 follows if we take
$$Y_T:=\{ u\in C([0,T];H^s): \n v\n <\infty \},$$
where the relations between $u$ and $v$, and between $u_0$ and $v_0$ are given by the equations \eqref{mt3.0} and \eqref{mt3.0a}, respectively.

\textit{Case $s>1$.}

By Theorem 1.6 in \cite{LP2009} there exist $T=T(\| u_0\|_{H^s})$ and a unique $u$ in the class defined by the conditions
\begin{align}
&u\in C([0,T]; H^s(\mathbb R^2)), \label{mtc1}\\
&\| D_x^s u_x\|_{L^\infty_x L^2_{yT}}+\|D_y^s u_x \|_{L^\infty_x L^2_{yT}}+\| u_x\|_{L^2_T L^\infty_{xy}}+ \|u \|_{L^2_x L^\infty_{yT}}   <\infty,\label{mtc2}
\end{align}

which is solution of the IVP \eqref{ZK}. Moreover, for any $T'\in(0,T)$ there exists a neighborhood $V$ of $u_0$ in $H^s$ such that the data-solution map $\tilde u_0\mapsto \tilde u$ from $V$ into the class defined by \eqref{mtc1} and \eqref{mtc2} with $T'$ instead of $T$ is Lipschitz.

Let $\{u_{0m} \}_{m\in\mathbb N}$ be a sequence in $C_0^\infty(\mathbb R^2)$ such that $u_{0m}\to u_0$ in $H^s(\mathbb R^2)$ and let $u_m\in C([0,T];H^\infty(\mathbb R))$ be the solution of the equation in \eqref{ZK} corresponding to the initial data $u_{0m}$. By Theorem 1.6 in \cite{LP2009}, $u_m \to u$ in $C([0,T];H^s(\mathbb R^2))$.

For $N\in\mathbb N$, let $w_N$ be the function defined in section 2.2. 

Let $p\equiv p_N$ be the function defined in $\mathbb R^2$ by
$$p(x,y):=(w_N(\sqrt{x^2+y^2}))^s.$$

We multiply the equation $\partial_t u_m+\partial_x \Delta u_m+u_m \partial_x u_m=0$ by $u_{m}p$, and for a fixed $t\in[0,T]$ we integrate in $\mathbb R^2$ with respect to $x$ and $y$, and use integration by parts to obtain
\begin{align*}
\dfrac d{dt}(u_m(t),u_m(t)p)=&-3(\partial_x u_m(t),\partial_x u_m(t) p_x)+(u_m(t),u_m(t)p_{xxx})+\dfrac 23 (u_m^3(t),p_x)\\
&-(\partial_y u_m(t),\partial_y u_m(t)p_x)-2(\partial_x u_m(t),\partial_y u_m(t) p_y)+(u_m(t),u_m(t)p_{xyy}),
\end{align*}

where $(\cdot, \cdot)$ denotes the inner product in $L^2(\mathbb R^2)$.

Integrating last equation with respect to the time variable in the interval $[0,t]$, we have
\begin{align}
\notag (u_m(t),u_m(t)p)=&(u_{0m},u_{0m} p)-3 \int_0^t (\partial_x u_m(t'),\partial_x u_m(t')p_x) dt'-\int_0^t (\partial_y u_m(t'),\partial_y u_m(t')p_x) dt'\\
\notag &-2\int_0^t (\partial_x u_m(t'),\partial_y u_m(t')p_y) dt'+\int_0^t (u_m(t'), u_m(t')(p_{xxx}+p_{xyy})) dt'\\
&+\dfrac 23 \int_0^t (u_m^3(t'),p_x) dt'.\label{mtc7}
\end{align}

Since $u_m\to u$ in $C([0,T];H^s(\mathbb R^2))$, $(s>1)$, and the weights $p$, $p_x$, $p_y$, $p_{xxx}+p_{xyy}$ are bounded functions, it follows from \eqref{mtc7}, after passing to the limit when $m\to\infty$, that
\begin{align}
\notag (u(t)&,u(t)p)\\
\notag =& (u_0,u_0 p)-3\int_0^t (\partial_x u(t'),\partial_x u(t')p_x) dt'-\int_0^t (\partial_y u(t'),\partial_y u(t') p_x)dt'\\
\notag &-2\int_0^t (\partial_x u(t'),\partial_y u(t')p_y) dt'+\int_0^t (u(t'),u(t')(p_{xxx}+p_{xyy}))dt'+\dfrac 23 \int_0^t (u^3(t'),p_x) dt'\\
\equiv & I+II+III+IV+V+VI. \label{mtc8}
\end{align}

Let us estimate the terms in the right-hand side of \eqref{mtc8}. First of all
\begin{align}
I\leq \|u_0 \|^2_{L^2((1+x^2+y^2)^{s/2} dx dy)}.\label{mtc9}
\end{align}

With respect to the term $II$, since $|p_x|\leq C w_N^{s-1} (\sqrt{x^2+y^2})$, we have
\begin{align*}
|II|\leq & C\int_0^t (\partial_x u(t'),\partial_x u(t') w_N^{s-1}(\sqrt{x^2+y^2})) dt' =C \int_0^t \|w_N^{(s-1)/2}(\sqrt{x^2+y^2}) \partial_x u(t') \|^2_{L^2(\mathbb R^2)} dt'\\
\leq & C \int_0^t  \| w_N^{(s-1)/2}(\sqrt{x^2+y^2}) J u(t')| \|^2_{L^2(\mathbb R^2)} dt',
\end{align*}

where $J:=(1-\Delta)^{1/2}$.

Using estimate \eqref{leinterpol} in Lemma \ref{interpol} with $w_N(\sqrt{x^2+y^2})$ we obtain
\begin{align}
\notag II\leq& C \int_0^t \| w_N (\sqrt{x^2+y^2})^{\frac{s-1}{s}\frac s2} J^{\frac 1s s} u(t') \|^2_{L^2} dt'\\
\leq & C \int_0^t \| w_N^{s/2}(\sqrt{x^2+y^2})  u(t') \|_{L^2}^{2(s-1)/s} \|J^s u(t') \|_{L^2}^{2/s} dt'\ .\label{mtc10}
\end{align}

Since $u\in C([0,T];H^s)$, then from \eqref{mtc10} it follows that, for $t\in[0,T]$,
\begin{align}
\notag |II|\leq & C \int_0^t \| w_N^{s/2} (\sqrt{x^2+y^2}) u(t')\|_{L^2} ^{2(s-1)/s} dt' \leq C  \int_0^t (1+\| w_N^{s/2}(\sqrt{x^2+y^2}) u(t')\|_{L^2}^2 )dt'\\
\leq & Ct+C\int_0^t (u(t'),u(t')w_N^s(\sqrt{x^2+y^2})) dt'= Ct+C\int_0^t (u(t'),u(t')p)dt'.\label{mtc11}
\end{align}

In a similar manner, taking into account that $|p_y|\leq C w_N^{s-1}(\sqrt{x^2+y^2})$, it can be seen that
\begin{align}
|III|, \, |IV| \leq & Ct +C\int_0^t (u(t'),u(t')p) dt'. \label{mtc12}
\end{align}

With respect to the term V, since $|p_{xxx}+p_{xyy}|\leq C w_N^{s-3}(\sqrt{x^2+y^2})$, we have
\begin{align}
|V|\leq C\int_0^t (u(t'),u(t') w_N^{s-3}(\sqrt{x^2+y^2})) dt'\leq C \int_0^t (u(t'),u(t')p) dt'.\label{mtc13}
\end{align}

In order to estimate $|VI|$, since $s>1$, we have
\begin{align}
\notag |VI|\leq & C \int_0^t \| u(t')\|_{L^\infty} (u(t'),u(t')w_N^{s-1}(\sqrt{x^2+y^2})) dt'\\
\leq & C \int_0^t \| u(t')\|_{H^s(\mathbb R^2)} (u(t'),u(t')w_N^{s}(\sqrt{x^2+y^2})) dt' \leq C \int_0^t (u(t'), u(t')p) dt'. \label{mtc14}
\end{align}

From equality \eqref{mtc8} and estimates \eqref{mtc9} to \eqref{mtc14} it follows that, for $t\in[0,T]$,
\begin{align*}
(u(t),u(t)p_N)\leq \| u_0\|^2_{L^2((1+x^2+y^2)^{s/2}dx dy)}+Ct+C\int_0^t (u(t'),u(t')p_N) dt'.
\end{align*}

Gronwall's inequality enables us to conclude that, for $t\in[0,T]$,
\begin{align}
(u(t),u(t)p_N)\leq \|u_0 \|^2_{L^2((1+x^2+y^2)^{s/2} dx dy)}+Ct+C\int_0^t (\| u_0\|^2_{L^2((1+x^2+y^2)^{s/2} dx dy)} +C t') e^{C(t-t')} dt'. \label{mtc15}
\end{align}

Passing to the limit in \eqref{mtc15} when $N\to\infty$ we obtain, for $t\in[0,T]$,
\begin{align}
\notag \| u(t)\|^2_{L^2((1+x^2+y^2)^{s/2} dx dy)}\leq & \| u_0\|^2_{L^2((1+x^2+y^2)^{s/2} dx dy)}+Ct\\
&+C \int_0^t (\| u_0\|^2_{L^2((1+x^2+y^2)^{s/2} dx dy)} +C t') e^{C(t-t')} dt', \label{mtc16}
\end{align}

which implies that $u\in L^\infty([0,T]; L^2((1+x^2+y^2)^{s/2} dx dy))$.

Proceeding as it was done in \cite{BJM2013}, it can be seen that $u\in C([0,T]; L^2((1+x^2+y^2)^{s/2}dx dy))$ and that if $\tilde u_m \in C([0,T];Z_{s,s/2})$ is the solution of the ZK equation, corresponding to the initial data $\tilde u_{m0}\to u_0$, where $\tilde u_{m0}$ in $Z_{s,s/2}$ when $m\to\infty$, then $\tilde u_m\to u_0$ in $C([0,T]; Z_{s,s/2})$. This fact, together with the continuous dependence proved in \cite{LP2009}, allow us to conclude that the assertion of theorem is true for the subspace $Y_T$ of $C([0,T];Z_{s,s/2})$ given by
$$Y_T=\{ u\in C([0,T];Z_{s,s/2}): \text{inequality \eqref{mtc2} holds} \}.$$

\end{proof}

\end{document}